\documentclass[onefignum,onetabnum]{siamart220329}
\usepackage{latexsym}
\usepackage{amssymb,amsbsy,amsmath,amsfonts,amssymb,amscd}
\usepackage{mathrsfs,hyperref}
\usepackage{epsfig}
\usepackage{xcolor}
\graphicspath{{./figures/}}

\usepackage{cleveref}
\usepackage{amsfonts}
\usepackage{graphicx}
\usepackage[caption=false]{subfig}
\usepackage{amsmath}
\usepackage{amssymb}
\usepackage{amsmath}
\usepackage{bm}
\usepackage{cases}
\usepackage{enumitem}
\usepackage{algorithm}
\usepackage{algorithmic}

\newcommand{\rd}{\mathrm{d}}

\newcommand{\ex}{\mathrm{ex}}

\newcommand{\norm}[1]{ \| #1 \|}

\newcommand{\Amat}{\mathsf{A}}
\newcommand{\Vmat}{\mathsf{V}}
\newcommand{\calG}{\mathcal{G}}
\newcommand{\half}{\frac{1}{2}}
\newcommand{\eps}{\varepsilon}

\theoremstyle{plain}\newtheorem{remark}{Remark}

\DeclareMathOperator*{\argmax}{arg\,max}
\DeclareMathOperator*{\argmin}{arg\,min}

\title{Differential-Equation Constrained Optimization With Stochasticity\thanks{Submitted to the editors \today.
\funding{Q.~Li is partially supported by ONR-N00014-21-1-214 and NSF-1750488. L.~Wang is partially supported by NSF grant DMS-1846854.
Y.~Yang is partially supported by ONR-N00014-24-1-2088.  Y.~Yang also acknowledges support from Dr.~Max R\"ossler, the Walter Haefner Foundation, and the ETH Z\"urich Foundation. }}}

\author{Qin Li\thanks{Department of Mathematics, University of Wisconsin-Madison, Madison, WI (\email{qinli@math.wisc.edu}).}
\and Li Wang\thanks{Department of Mathematics, University of Minnesota Twin Cities, Minneapolis, MN (\email{liwang@umn.edu}).}
\and Yunan Yang\thanks{Department of Mathematics,  Cornell University,  Ithaca, NY (\email{yunan.yang@cornell.edu})}}

\headers{}{Qin Li, Li Wang, and Yunan Yang}

\date{\today}

\begin{document}
\maketitle
\begin{abstract}
Most inverse problems from physical sciences are formulated as PDE-constrained optimization problems. This involves identifying unknown parameters in equations by optimizing the model to generate PDE solutions that closely match measured data. The formulation is powerful and widely used in many sciences and engineering fields. However, one crucial assumption is that the unknown parameter must be deterministic. In reality, however, many problems are stochastic in nature, and the unknown parameter is random. The challenge then becomes recovering the full distribution of this unknown random parameter. It is a much more complex task. In this paper, we examine this problem in a general setting. In particular, we conceptualize the PDE solver as a push-forward map that pushes the parameter distribution to the generated data distribution. This way, the SDE-constrained optimization translates to minimizing the distance between the generated distribution and the measurement distribution. We then formulate a gradient-flow equation to seek the ground-truth parameter probability distribution. This opens up a new paradigm for extending many techniques in PDE-constrained optimization to that for systems with stochasticity.
\end{abstract}


%65M32: Numerical methods for inverse problems for initial value and initial-boundary value problems involving PDEs
%49Q22: Optimal transportation 
% 65M75  	Probabilistic methods, particle methods, etc. for initial value and initial-boundary value problems involving PDEs
% 65K10 Numerical optimization and variational techniques

\section{Introduction}
We study the problem of inferring the random parameters in a differential equation (DE). In particular, we ask:

\medskip
\begin{center}
    \emph{How to recover the distribution of an unknown random parameter in a DE from that of measurements?}
\end{center}

\medskip
The problem comes from the fact that many differential equations are equipped with parameters that are random in nature. Even with a fixed set of boundary/initial conditions and measuring operators, the measurements are nevertheless random, with the randomness coming from different realizations of the parameter. Aligned with other inverse problems that infer unknown parameters from measurements, we now target to recover the distribution of the parameters from the distribution of measurements. Therefore, the problem under study is the stochastic extension of the deterministic inverse problems, particularly the partial differential equations (PDE)-constrained optimization.

In the deterministic and finite-dimensional setting, many problems are formulated as
\begin{equation}\label{eqn:G}
y =\mathcal{G}(u)\,,\quad\calG:\mathcal{A}\subset\mathbb{R}^m\to\mathcal{R}\subset\mathbb{R}^n\,.
\end{equation}
This is to feed a system, described by $\mathcal{G}$, with an $m$-dimensional parameter $u$, to produce the data $y\in\mathbb{R}^n$. The set $\mathcal{A}$ stands for all the admissible parameters, and $\mathcal{R}$ is the range of $\mathcal{G}$. The inverse problem is to revert the process. The model $\mathcal{G}$ is still given, but the goal is to infer the parameter $u$ using the measured data $y^\ast$. Potentially, $y^\ast$ contains measurement error.

Inverse problems present many challenges. To determine $u$ uniquely from $y^\ast$, one requires $\mathcal{G}$ to be injective and $y^\ast$ to lie within the range of $\mathcal{G}$. However, as elegantly summarized in~\cite{Stuart}, this is not typically the case in practice. Therefore, various techniques have been introduced to approximate the inference, and optimization is one of the most popular numerical strategies. Specifically, the goal is to find the configuration of $u$ that best matches the data $y^\ast$. Denoting $D$ as the metric or divergence applied to the data range $\mathcal{R}$, we search for $u^\ast$ such that:
\begin{equation}\label{eqn:argmin_pde}
u^\ast {\in}\,\text{argmin}\; D(y^\ast,\mathcal{G}(u))=:L(u)\,.
\end{equation}
Although this minimization formulation solves the original inverse problem when $y^\ast$ falls within the range of $\mathcal{G}$, it leaves the problem of ``invertibility'' unresolved. This is because this formulation may have many, or even infinitely many, minimum points. Furthermore, this formulation introduces an additional layer of difficulty regarding ``achievability'': even if the landscape of the loss function guarantees a unique minimizer, conventional optimization strategies may not be able to find it. Nonetheless, among the many optimization algorithms available, it is common to use gradient descent (GD) or other first-order optimization methods to search for the optimizer if a good initial guess is provided. In the continuous setting, the iteration number in GD is transformed to the time variable $s$, and the corresponding gradient flow is written as follows:
\begin{equation}\label{eqn:GD}
\frac{\rd u}{\rd s}=-\alpha\nabla_uL\,,
\end{equation}
where $\alpha$ signifies the rate and can be adjusted according to the user's preferences. When $u$ is a function, $\nabla_uL$ denotes the functional derivative of $L$ with respect to $u$. 

Many real-world problems can be formulated using~\eqref{eqn:argmin_pde}, such as PDE-constrained optimization problems. In such problems, $\mathcal{G}$ is the map induced by the underlying PDE, which maps the PDE parameter to the measurements. For instance, consider $f_i$ as the solution to a PDE characterized by the PDE operator $\mathcal{L}(u)$ that is parameterized by the coefficient $u$ with the $i$-th source term $S_i$, and let $\mathcal{M}_j$ denote its $j$-th measurement operator, $1\leq i \leq I$, $1\leq j \leq J$. Then, the measurement is an $I\times J$ matrix with its $ij$-th entry as $\left(\mathcal{G}(u)\right)_{ij}=\mathcal{M}_j(f_i)$. In this case, \eqref{eqn:argmin_pde} can be naturally represented as a PDE-constrained optimization problem:
\begin{equation}\label{opt:pde}
u^\ast {\in} \,\text{argmin} \frac{1}{IJ}\sum_{ij}|\mathcal{M}_j(f_i)-d_{ij}|^2\quad \text{s.t.}\quad  \mathcal{L}(u) [f_i] = S_i\,.
\end{equation}
As with the more general formulation~\eqref{eqn:argmin_pde}, 
the PDE-constrained optimization problem~\eqref{opt:pde} may not always be solvable. However, if we attempt to solve it using GD or its flow formulation~\eqref{eqn:GD}, the updating formula necessitates the computation of the gradient, which is typically done by solving both the forward and the adjoint equations: $\mathcal{L}(u) [f_i] =S_i$ and $\mathcal{L}^\ast(u) [g_j] = \psi_j$, with $ \psi_j$ determined by choice of the objective function. 
Over the years, this problem has attracted significant interest from various scientific communities, and many aspects of gradient computation have been investigated, see the book~\cite{hinze2008optimization}. 

Building upon this framework, we are interested in a class of problems where the unknown parameter $u$ is stochastic. Since $u$ is known to be random a priori, we aim to infer its distribution, denoted by $\rho_u$. Due to the inherent randomness in $u$, the measurement is also random. By using~\eqref{eqn:G}, we can write the distribution of data by regarding $\mathcal{G}$ as a push-forward map:
\[
\rho_y =\mathcal{G}_{\sharp}\rho_u\,.
\]
Consequently, the inverse problem is to infer the ground-truth distribution of $u$, denoted by $\rho_u^\ast$, using the measured data distribution, denoted by $\rho^\ast_y$. Denoting $D$ as the metric or divergence used {to measure data discrepancy} in the space of probability measures, we aim to find $\rho_u^\ast$ such that its push-forward matches $\rho^\ast_y$ as closely as possible. Similar to~\eqref{opt:pde}, $\mathcal{G}$ can be induced by a PDE, which leads to the following problem of \emph{DE-constrained optimization with stochasticity}:
\begin{equation}\label{eqn:argmin_sde}
\rho_u^\ast {\in} \text{argmin}\; D(\mathcal{G}_{\sharp}\rho_u, \rho^\ast_y)\,.
\end{equation}

As with the PDE-constrained optimization, whether the optimizer is unique and whether a simple gradient-based method can achieve the optimizer are unknown. The answers to these questions are rather problem-specific, and this goal is beyond the scope of a single paper. Instead, we strive to make the following contributions:
\begin{itemize}
    \item[1.]We will provide a recipe for a first-order gradient-based solver for~\eqref{eqn:argmin_sde}. Using different metrics to define the ``gradient'' for probability measures and different definitions of the distance/divergence function $D$, we can generate the corresponding gradient flows on this metric to move $\rho_u$ along.
    \item[2.] When $D$ is the Kullback--Leibler (KL) divergence, and the underlying metric for the probability space is the 2-Wasserstein distance, we will provide a particle method to simulate the gradient-flow equation. The updating formula for the particle is a combination of a forward and an adjoint solver pair.
    \item[3.] In the linear case, we study the well-posedness in both under-determined and over-determined scenarios and draw a relation to their counterparts in the deterministic setting.
\end{itemize}

Our formulation appears to be closely related to several well-established research topics, which we will discuss in Section~\ref{sec:compare}, including their connections and differences. This will be followed by a gradient flow formulation using different metrics for various choices of distance or divergence $D$ in~\eqref{eqn:argmin_sde}. In Section~\ref{sec:GF}, we also discuss the associated particle method, which solves the Wasserstein gradient flow for the KL distance. Section~\ref{sec:theory} presents the available theoretical results when the underlying map $\mathcal{G}$ is linear. We have discovered that the theoretical results correspond one-to-one to the over/under-determined linear system under the deterministic scenario. Finally, numerical evidence is presented in Section~\ref{sec:numerics} to support our findings, followed by the conclusion in~\Cref{sec:conclude}.

\section{Comparisons with Other Subjects}\label{sec:compare}
Several other research fields in applied mathematics are closely related to the problem we propose to study in the Introduction. Bayesian inverse problem~\cite{Stuart}, for instance, is a field that uses probabilistic models to infer unknown parameters of a system from observed data. 
Another related field is density estimation~\cite{silverman1986density}  that focuses on estimating the probability density function of a random variable from a set of observations. In the following subsections, we will conduct a comprehensive review of these related fields, discuss how they are related to our problem, and point out some key differences between them and our proposed approach.

\subsection{Bayesian Inverse Problem}
We now draw the connection to the Bayesian inverse problem \cite{marzouk2007stochastic,Stuart} and the associated sampling problem~\cite{garbuno2020interacting,chen2023gradient,summary_da,kim2023forwardbackward}. Again starting from \eqref{eqn:G}, it considers the scenario where the measurement data is corrupted by noise. Most commonly, the observed data $y$ is assumed to be the sum of the true data and a random noise term
\begin{equation*} 
y = \mathcal{G} (u) + \eta\,,\quad u \in \mathcal{A}\subset\mathbb{R}^m\,,
\end{equation*}
where the additive noise is assumed to be Gaussian, i.e., $\eta \sim \mathcal{N}(0,\Gamma)$. Additionally, a prior distribution $\rho^{\text{prior}}(u)$ for the parameter $u$ is given, often assumed to be Gaussian as well. Then the goal is to find the most efficient way to sample from the posterior given by Bayes' theorem:
\begin{align} \label{Bay1}
  \rho^{\text{post}} (u) \propto \mathbb{P}(y|u ) \rho^{\text{prior}}(u )\,, 
\end{align}
where $\mathbb{P}(y|u )$ is the likelihood function, whose concrete form is determined by the noise assumption for the data $y$. Since $\eta \sim \mathcal{N}(0,\Gamma)$, the likelihood function is given by
\[
\mathbb{P}(y|u ) \propto e^{-\half |y - \calG(u)|_\Gamma^2 }\,.
\]

Although the ultimate objective in both the Bayesian inverse problem and our framework~\eqref{eqn:argmin_sde} is to identify the distribution of the parameter $u$, the sources of randomness in our formulation are \textit{fundamentally different} from those in the Bayesian framework. Specifically, in the Bayesian framework, uncertainty arises due to noise in the measurement process of obtaining $y$ while the true data $\calG(u)$ is assumed to be deterministic. This means if {the measurement is entirely precise and $\calG$ is invertible}, the posterior distribution would be a delta measure, meaning that $u$ would be deterministic.

In contrast, in our case, the true data is a probability distribution by itself. Under a deterministic forward map $\calG$, the randomness of the parameter is an inherent feature of the system being modeled, and even the precise ``reading'' $y$ would nevertheless lead to a data distribution.

With the prior $\rho_0 =  \mathcal{N}(m_0, \Sigma_0)$ and the noise assumption on the data, the posterior distribution is of the form $\rho(u):=\frac{1}{Z} e^{-V(u)}$ with
$V(u) = \half |y - \calG(u)|_\Gamma^2 + \half |u - m_0|_{\Sigma_0}^2$. One way to sample from the posterior~\eqref{Bay1} is to use the Langevin dynamics~\cite{robert1999monte}, where  a set of particles $\{u^{(j)}\}$ are evolved  by
\begin{align*} 
\rd u^{(j)} (t) = -\nabla V(u^{(j)}) \rd t + \sqrt{2} \rd W(t)\,, \quad u^{(j)}(0) \sim \rho^{\text{prior}}(u)\,.
\end{align*}
In the infinite time horizon, i.e., when $t \rightarrow \infty$, $\{u^{(j)} (t)\}$ will be samples from the posterior $\rho^{\text{post}}(u)$. 

The Bayesian inverse problem also bares a variational formulation. The posterior $\rho(u)$ can be characterized as a distribution that minimizes the Kullback--Leibler (KL) divergence~\cite{zellner1988optimal}, i.e., 
\begin{align} \label{vaB}
\rho^{\text{post}}(u) \in \argmin_{\pi(u)} KL\left(\pi(u)| \mathbb{P}(y|u ) \rho^{\text{prior}}(u ) \right) \,.
\end{align}
Therefore, $\rho^{\text{post}}(u)$ can be obtained by evolving the Wasserstein gradient flow of the divergence in~\eqref{vaB} to equilibrium. The Wasserstein metric and its corresponding energy function, such as the KL divergence, can take different forms in the formulation \cite{garcia2020bayesian}. It is apparent from both Bayesian formulation~\eqref{vaB} and our formulation \eqref{eqn:argmin_sde} that, although both attempt to match the target distribution, the target distribution and source distribution in each approach differs.

\subsection{Density Estimation}
Another related field to our formulation is density estimation. Density estimation is a statistical technique used to estimate the probability density function of a random variable from observed samples~\cite{silverman1986density,sheather2004density,wang2022minimax}. It also finds great use in artificial intelligence such as generative modeling~\cite{albergo2023stochastic,song2019generative,lee2023convergence,chen2023sampling}. One common approach to density estimation is to use kernel density estimation~\cite{chen2017tutorial}, which involves convolving a set of basis functions (typically kernels) with the observed samples to estimate the underlying density function. The choice of kernel function and bandwidth parameter can affect the accuracy and smoothness of the resulting estimate. Other methods for density estimation include histogram-based methods, parametric models (e.g., using a normal distribution), flow-matching type methods, and non-parametric models (e.g., using a mixture of distributions). 

Density estimation is mostly studied for a parameterized distribution where the goal is to estimate the parameters of a specific distribution rather than the entire density function. This is commonly done using maximum likelihood estimation, a variational approach~\cite{welling2011bayesian}. Given a set of observations $y_1, y_2, ..., y_N$, and a true but unknown probability density function $p(y)$, we seek to estimate the parameter $\theta$ in the density function $q(y;\theta)$ that minimizes the KL divergence between $q(y;\theta)$ and $p(y)$. That is,
\begin{align*}
    \theta^* &= \argmin_{\theta} \text{KL} \left( p(y) | q(y;\theta) \right) = \argmin_{\theta} \mathbb{E}_{p(y)} \left[\log p(y) - \log q(y;\theta) \right] \nonumber \\
    &=   \argmax_{\theta} \mathbb{E}_{p(y)} \left[ \log q(y) \right] \approx \argmax_{\theta} \frac{1}{N}\sum_{i=1}^N \log q(y_i;\theta)\,, %\label{eq:density_est}
\end{align*}
where the last two terms correspond to the so-called maximum log-likelihood estimation. 

In our framework~\eqref{eqn:argmin_sde}, when the push-forward map $\mathcal{G} = I$, the identity map, then the problem reduces to a classical density estimation problem. Otherwise, we are performing a density estimation for $u$ not from samples of $u$, but samples of $\rho_y = \mathcal{G}\sharp \rho_u$ where $u\sim \rho_u$. The problem in~\eqref{eqn:argmin_sde} combines both the density estimation aspect (from samples of $y$ to the density of $y$) and the inversion part (from the density of $y$ to the density of $u$).

\section{Gradient Flow Formulation and Particle Methods}\label{sec:GF}
Denote $\mathcal{P}_2({\mathcal{A}})$ the collection of all probability measures supported on $\mathcal{A}$, the admissible set, with finite second-order moments. We endow $\mathcal{P}_2({\mathcal{A}})$ with a metric $d_g$. Being confined to the set $\mathcal{P}_2(\mathcal{A})$ equipped with a metric $d_g$, the variational formulation~\eqref{eqn:argmin_sde} for a DE-constrained optimization problem becomes
\begin{align}\label{eqn:argmin_sde3}
    \rho_u^\ast= \argmin_{\rho_u\in \left( \mathcal{P}_2(\mathcal{A}), d_g\right)} E(\rho_u)  := D(\rho_y\,, \rho^\ast_y)=D(\calG_{\sharp}\rho_u, \rho^\ast_y)\,.
\end{align}

This is to find the optimal distribution $\rho^\ast_u\in\mathcal{P}_2(\mathcal{A})$, which, upon being pushed forward by $\mathcal{G}$, is the closest to the probability distribution of the data: $\rho^\ast_y$, measured by {the data discrepancy} $D$. The map $\calG$ is then given by the deterministic forward operator that maps a fixed parameter configuration to the measurement.

Similar to the fact that the gradient flow~\eqref{eqn:GD} is used to solve the deterministic optimization problem~\eqref{eqn:argmin_pde}, one can run gradient-descent type algorithms on the space of probability measures to solve a variational problem defined over the probability space such as~\eqref{eqn:argmin_sde3}.
Since such an optimization problem is over an infinite-dimensional space, gradient-based algorithms are particularly attractive in terms of computational cost.  Using gradient-based algorithms for~\eqref{eqn:argmin_sde3} amounts to updating $\rho_u$ based on the gradient direction of $E(\rho_u)$. The precise definition of the ``gradient'' here relies on the metric $d_g$ that the underlying probability space $\mathcal{P}_2(\mathcal{A})$ is equipped with. 

We want to address that different pairs of $(d_g, E)$ yield different gradient flow formulations to update $\rho_u$. We will focus on a few concrete examples in the following subsections.

\subsection{The Wasserstein Gradient Flow Strategy}
First, we consider $d_g$ as the quadratic Wasserstein metric ($W_2$). We first define $W_2$ using the Kantorovich formulation of the optimal transportation problem.
\begin{definition}[Kantorovich formulation]
Let $(M,d)$ be a metric space. The quadratic Wasserstein metric between two probability measures $\mu$ and $\nu$ defined on $M$ with finite second-order moments is
\[
W_2(\mu, \nu) = \left( \inf_{\gamma \in \Gamma(\mu, \nu)} \int_{M\times M} \rd(x, y)^2 \rd \gamma(x,y) \right)^{1/2}\,,
\]
where $\Gamma(\mu, \nu)$ is the set of all coupling for $\mu$ and $\nu$. A coupling $\gamma$ is a joint probability measure on $M \times M$ whose marginal distributions are $\mu$ and $\nu$, respectively. That is,
\begin{equation*}
% \begin{aligned}
\int_M \gamma(x, y) \,\rd y = \mu(x)\,,\quad \int_M \gamma(x, y) \,\rd x = \nu(y)\,.
% \end{aligned}
\end{equation*}
\end{definition}
Throughout our paper, we set $\mathcal{A}\subset M = \mathbb{R}^m$ with $d$ being the Euclidean distance. {Equipped with the 2-Wasserstein metric, the Wasserstein gradient flow equation for $E(\rho_u)$ writes~\cite{ambrosio2005gradient,santambrogio2015optimal}:}
\begin{equation} \label{eqn:GF}
\partial_t\rho_u =- \nabla_{W_2} E(\rho_u) =  \nabla_u\cdot\left(\rho_u\nabla_u \frac{\delta E}{\delta \rho_u}\right)\,.
\end{equation}

This gradient flow is guaranteed to descend the energy. To see this, we multiply $\frac{\delta E}{\delta\rho_u}$ on both sides:
\begin{align*}
    \frac{\rd}{\rd t} E(\rho_u) =  \int \partial_t \rho_u \frac{\delta E}{\delta \rho_u} \rd u =  \int \nabla_u\cdot\rho_u\nabla_u\left(\frac{\delta E}{\delta \rho_u}\right)\frac{\delta E}{\delta \rho_u} \rd u = - \int \rho_u(u) \bigg|\nabla_u \frac{\delta E}{\delta \rho_u}\bigg|^2\rd u \leq 0\,.
\end{align*}
Immediately, we see from the fact that the right-hand side is negative, $E(\rho_u)$ decays in time, which implies that the equilibrium $\rho_u^\infty$ should satisfy 
\begin{equation} \label{0312}
\nabla_u \frac{\delta E}{\delta \rho_u^\infty } = 0 \quad \text{on the support of}~ \rho_u^\infty\,.
\end{equation}

For the given specific form of $E(\rho_u)$ in~\eqref{eqn:argmin_sde3}, we find an explicit formulation for $\frac{\delta E}{\delta \rho_u}$. This can be done through the standard technique from the calculus of variation. According to the definition of the Fr\'echet derivative, we perturb $\rho_u$ by $\delta \rho_u$ where $\int  \delta \rho_u \rd u = 0$. Then we have
\begin{align} \label{0301}
\lim_{\norm{\delta \rho_u}_2 \rightarrow 0} [E(\rho_u + \delta \rho_u) - E(\rho_u) ] 
= \int \frac{\delta E}{\delta {\rho_u}} \delta \rho_u \rd u \,.
\end{align}
Since $y=\calG(u)$, it follows that
\[
 \rho_y=\calG_{\sharp}\rho_u\,,\quad\text{and}\quad \delta\rho_y=\calG _{\sharp}\delta \rho_u\,.
\]
Substituting the definition $E(\rho_u) = D(\calG _{\sharp} \rho_u, \rho^*_y)$, we find
\begin{align} \label{0302}
    & E(\rho_u + \delta \rho_u) - E(\rho_u) \nonumber 
    \\ =& D(\rho_y + \delta \rho_y, \rho_y^*) - D(\rho_y, \rho_y^* ) \nonumber 
    \\ = & \int \frac{\delta D}{\delta \rho_y} (y) \delta \rho_y (y) \rd y + \text{higher order terms} \nonumber 
    \\ = &\int \frac{\delta D}{\delta \rho_y} (\calG(u)) \delta \rho_u(u) \rd u + \text{higher order terms}, 
\end{align}
where the last equality uses the definition of a push-forward map. That is, if $f = T_{\sharp}g$ by $y=T(x)$, then for any measurable function $F$ and $\Omega$ in the support of $f$, 
$$\int_{x \in T^{-1}(\Omega)}  F(T(x)) g(x)\rd x  = \int_{y \in \Omega} F(y) f(y) \rd y.$$
Comparing \eqref{0301} with \eqref{0302}, we have:
\begin{align}\label{eqn:Frechet}
    \frac{\delta E}{\delta \rho_u}(u) =   \frac{\delta D}{\delta \rho_y} \circ \mathcal{G} (u)\,.
\end{align}

This means the Fr\'echet derivative of $E$ on $\rho_u$ is that of $D$ on $\rho_y$ composing with the push-forward map $\mathcal{G}$. As a concrete example, we consider $D$ in \eqref{eqn:argmin_sde3} to be the KL divergence, namely:
\begin{align}\label{eqn:KL_obj}
    D(\rho_y, \rho_y^*) = \int \rho_y \log \frac{\rho_y}{\rho_y^*} \rd y\,,\quad\text{and}\quad \frac{\delta D}{\delta \rho_y} = \log \rho_y - \log \rho_y^* + 1 \,.
\end{align}
Combining this with~\eqref{eqn:GF} and~\eqref{eqn:Frechet}, we finalize the evolution equation:
\begin{equation}\label{eqn:GF_KL}
\partial_t\rho_u  = \nabla_u\cdot\left(\rho_u\nabla_u\left( \log \frac{\rho_y}{ \rho^\ast_y}(\calG(u)) \right)\right)\,.
    % \partial_t\rho_u +\nabla_u\cdot\left(\rho_u\nabla_u\left( \log \rho_y\circ \calG   - \log  \rho^\ast_y \circ \calG \right)\right)=0\,.
\end{equation}
Throughout the paper, it is the convention to use the quadratic Wasserstein metric (i.e., $d_g = W_2$) to equip $\mathcal{P}_2(\mathcal{A})$ if not specifically mentioned otherwise. Similarly, $D=\text{KL}$ is used as the convention to define the cost functional $E(\rho_u)$.

\subsection{Other Possible Metrics}
The probability space can be metricized in various ways, and there exist several data misfit functions $D$ to serve as the data  discrepancy, used in place of the KL divergence. In this subsection, we provide a few examples to demonstrate the breadth and versatility of the proposed framework.

\subsubsection*{Wasserstein Gradient Flow of the  Wasserstein-Based Objective Function}
One choice is to set $D$ based on the Wasserstein metric with the cost function $\mathsf c$ while setting $d_g = W_2$. More precisely, let
\begin{equation}\label{eqn:discrepancy_w2}
  D(\rho_y, \rho_y^*) := \inf_{\gamma \in \Gamma(\rho_y, \rho_y^*)} \int_ {M \times M} \mathsf c(x,y) \rd \gamma \,,
\end{equation}
where $\Gamma(\rho_y, \rho_y^\ast)$ is the set of all coupling of $\rho_y$ and $\rho_y^*$, and $\mathsf c $ is a cost function, e.g., $\mathsf c (x,y) = |x-y|^2$. Then by the Kantorovich duality~\cite{villani2021topics}, it has the following equivalent form
\begin{align} \label{eq:W2_dual}
   D(\rho_y, \rho_y^*) = \sup_{(\phi, \psi) \in \Phi(\rho_y, \rho_y^*)} \int_M \phi(x) \rho_y(x) \rd x + \int_M \psi(y) \rho_y^*(y) \rd y\,,
\end{align}
where $\Phi(\rho_y, \rho_y^*)$ is the set of pairs $(\phi, \psi)$ such that $\phi(x) + \psi(y) \leq \mathsf c(x,y)$ for all $(x,y) \in M \times M$. If we denote $(\phi^*, \psi^*)$ the maximizing pairs, also referred to as the Kantorovich potentials, then we have~\cite{villani2021topics}
\begin{align} \label{eq:W2_dual_Frechet}
    \frac{\delta D}{\delta \rho_y} = \phi^*(x)\,.
\end{align} 
Following~\eqref{eqn:GF} and~\eqref{eqn:Frechet}, the corresponding Wasserstein gradient flow is
\begin{align} \label{eq:GF_W2}
    \partial_t \rho_u = \nabla_u \cdot \left( \rho_u \nabla_u \phi^*\left(\calG (u) \right) \right)\, .
\end{align}

\subsubsection*{Hellinger Gradient Flow of the $\chi^2$ Divergence} 
While the Wasserstein metric is typically used for its simplicity in bridging particle systems and the underlying flows, one can also equip $\mathcal{P}_2(\mathcal{A})$ with other metrics (i.e., $d_g$ in~\eqref{eqn:argmin_sde3}) over the probability space. Another example is the Hellinger distance defined below.
\begin{definition}[The Hellinger distance]
Consider two probability measures $P$ and $Q$ both defined on a measure space $M$ that are absolutely continuous with respect to an auxiliary measure $\mu$, i.e., 
\begin{align*}
    P(\rd x) = p(x) \mu(\rd x), \quad 
    Q(\rd x) = q(x) \mu(\rd x)\,.
\end{align*}
The Hellinger distance between $P$ and $Q$ is
$
    H^2(P,Q) = \half \int_M \left( \sqrt{p(x)} - \sqrt{q(x)} \right)^2 \mu(\rd x)\,.
$
\end{definition}

Following \cite{lindsey2022ensemble}, we consider the gradient flow when $D$ is the chi-squared ($\chi^2$) divergence and $E(\rho_u)$ is determined correspondingly. More specifically, we have
\begin{align} \label{chi2}
    D(\rho_y, \rho_y^*) = \chi^2 (\calG_{\sharp} \rho_u, \rho_y^* ) = \int \frac{(\calG_{\sharp} \rho_u)^2 }{\rho_y^*} \rd y - 1 \,.
\end{align}
Then the gradient flow of \eqref{chi2} with respect to the Hellinger distance can be derived via the so-called JKO scheme~\cite{jordan1998variational}.
That is,
\begin{align} \label{JKO0}
    \partial_t \rho_u = \lim_{\eps \rightarrow 0} \frac{\rho_u^\eps - \rho_u}{\eps} \,,
\end{align}
where 
$
    \rho_u^\eps \in \argmin_{\int \tilde \rho_u \rd u =1} \left\{ D(\calG_{\sharp} \tilde \rho_u, \rho_y^*) + \frac{1}{2\eps} H^2(\tilde \rho_u, \rho_u)\right\} \,.
$
Then the optimality condition yields the following relation
\begin{align} \label{JKO2}
    \rho_u^\eps= \rho_u - 4 \eps \rho_u \left[ \frac{2\rho_y}{\rho_y^*}(\calG(u)) -\lambda \right] \,,
\end{align}
where $\lambda$ is the Lagrangian multiplier to make sure that $\rho_u^\eps$ integrates to one. Hence, 
$\lambda = \int  \frac{2\rho_y}{\rho_y^*} (\calG (u)) \rho_u \rd u$.
Plugging \eqref{JKO2} into \eqref{JKO0}, we get
\begin{align*}
    \partial_t \rho_u = 8 \rho_u \left[ \int  \frac{\rho_y}{\rho_y^*} (\calG (u)) \rho_u \rd u -  \frac{\rho_y}{\rho_y^*} (\calG (u))  \right]\,.
\end{align*}

\subsubsection*{Kernelized Wasserstein Gradient Flow}
The so-called Stein Variational Gradient Descent (SVGD) can be seen as the kernelized Wasserstein gradient flow of the KL divergence~\cite{liu2017stein}. Through an equivalent reformulation of the same dynamics, SVGD can also be regarded as the kernelized Wasserstein
gradient flow of the $\chi^2$
divergence but with a different kernel~\cite{chewi2020svgd}. In~\cite{duncan2019geometry}, SVGD is formulated as the true gradient flow of the KL divergence under the newly defined Stein geometry. That is, $d_g$ is a metric based on the Stein geometry, while $E$ is decided by setting $D$ as the KL divergence in~\eqref{eqn:argmin_sde3}.

\subsection{Particle Method}
One significant advantage of using the gradient flow formulation~\eqref{eqn:GF} with the underlying metric being $W_2$ is the ease of translating the produced PDE formulation, such as~\eqref{eqn:GF} and~\eqref{eqn:GF_KL}, to a particle formulation. This feature allows an easy implementation of numerical schemes. According to~\eqref{eqn:GF_KL}, each particle i.i.d. drawn from $\rho_u$ should evolve by descending along its negative velocity field:
\[
\frac{\rd}{\rd t}u(t)  = -\nabla_u\left.\frac{\delta E}{\delta\rho_u}\right|_{\rho_u(t)}(u(t)) = -\nabla_u\mathcal{G}^\top |_{u(t)} \xi(y(u(t)))\,,
\]
where we used the definition of
\begin{equation}\label{eqn:def_xi}
\xi(y(u(t))):=\nabla_y\left.\frac{\delta D}{\delta\rho_y}\right|_{\mathcal{G}\sharp\rho_u(t)}(y(u(t)))\,.
\end{equation}
When $\mathrm{KL}$ divergence is used, the expression for $\xi$ can be made more explicit, and thus:
\begin{equation}\label{eqn:GF_u}
\frac{\rd}{\rd t}u  = -\nabla_u\left( \log \frac{\rho_y}{ \rho^\ast_y}(\calG(u))\right)= -\left.\nabla_u\calG^\top \right|_{u(t)}\nabla_y\log\left(\rho_y/\rho_y^\ast\right)(y(t))\,,\,\,\text{where} \,\, y(t) = \mathcal{G}(u(t))\,,
\end{equation}
where $\nabla_u \calG\big|_{u_j(t)} \in \mathbb{R}^{n\times m}$ is a Jacobian matrix of $\mathcal{G}$. Multiplying the Jacobian $\frac{\rd y}{\rd u}=\nabla_u\calG$ on both sides, we have
\begin{equation}\label{eqn:GF_y}
\frac{\rd}{\rd t}y =  - \left.\nabla_u\calG\right|_{u(t)} \left.\nabla_u\calG\right|_{u(t)}^\top \nabla_y\log\left(\rho_y/\rho_y^\ast\right)\,,\quad\text{where}\quad y\sim\rho_y\,.
\end{equation}

Note that in practice, these two formulas cannot be executed because $\rho_u$ and $\rho_y$ are unknown. Therefore, one needs to replace it with its numerical approximation. To be more precise, let $u_i$ be a list of particles drawn from $\rho_u$, then we write $\rho^N_u$ as an empirical distribution approximating $\rho_u$, i.e.,
\begin{equation}\label{eqn:ensemble_u}
\rho_u\approx  \rho^N_u=\frac{1}{N}\sum_{j=1}^N\delta_{u_j}\,,\quad\text{and thus}\quad \rho^N_y = \calG_{\sharp}\rho^N_u = \frac{1}{N}\sum_{j=1}^N\calG_{\sharp}\delta_{u_j}=\frac{1}{N}\sum_{j=1}^N\delta_{y_j}\,.
\end{equation}
In this ensemble version, all the particles $u_j$ evolve according to:
\begin{equation}\label{eqn:particle_method}
\begin{aligned}
\partial_t u_j = - \nabla_u \left.\frac{\delta E}{ \delta \rho_u}\right|_{\rho^N_u} (u_j)&=  - \nabla_u \left[\log \rho^N_y(\calG(u_j)) - \log \rho^\ast_y(\calG(u_j)) \right]\\
&=
\frac{\nabla_u \rho^{\ast}_y(\calG(u_j)}{\rho^{\ast}_y(\calG(u_j)}  - \frac{\nabla_u \rho^N_y(\calG(u_j)}{\rho^N_y(\calG(u_j)} \\
& = \nabla_u \calG^\top \big|_{u_j}  \underbrace{\left( 
\frac{\nabla_y \rho^{\ast}_y(y_j)}{\rho^{\ast}_y(y_j)} - \frac{\nabla_y \rho^N_y(y_j)}{\rho^N_y(y_j)} \right)}_{\xi_j=\nabla_y\left.\frac{\delta D}{\delta\rho_y}\right|_{\rho^N_y}(y_j)}\,,
\end{aligned}
\end{equation}
where $y_j = \calG(u_j) \in \mathbb{R}^n$ and they both evolve in time.

It is easy to see that when $\rho_y$ is the ensemble distribution~$\rho^N_y$ defined in~\eqref{eqn:ensemble_u}, $\xi_j$ is not well-defined.
In particular, the singularity induced by the Dirac deltas can be numerically inaccessible. In simulations, one has to approximate $\rho_y^N$ using probability density functions to implement~\eqref{eqn:particle_method}. One option is to use kernel density estimation with an isotropic Gaussian kernel~\cite{chen2017tutorial,calvello2022ensemble}. That is,
\[
\rho^N_y\approx\frac{1}{N}\sum_{j=1}^N\phi^\epsilon(y-y_j)\,,\quad \phi^\epsilon(y)=\frac{1}{(2\pi\epsilon)^{n/2}}\exp\left(-\frac{y^2}{2\epsilon}\right)\,.
\]
This formulation leads to:
\[
\frac{\nabla_y \rho^N_y}{\rho^N_y}(y)=-\frac{1}{\epsilon}\frac{\sum_{j} (y-y_j)e^{-(y-y_j)^2/2\epsilon}}{\sum_{j} e^{-(y-y_j)^2/2\epsilon}}\,,\quad \forall y\,,
\]
and when plugged into~\eqref{eqn:particle_method}, it gives the final particle method:
\begin{equation}\label{eqn:particle_KL}
\partial_t u_j ={ \left( \nabla_u \calG\big|_{u_j} \right)^\top \xi_j}\,,\quad\text{with}\quad \xi_j =  
\frac{\nabla_y \rho^{\ast}_y(y_j)}{\rho^{\ast}_y(y_j)} +\frac{1}{\epsilon}\frac{\sum_{i} (y_j-y_i)e^{-(y_j-y_i)^2/2\epsilon}}{\sum_{i} e^{-(y_j-y_i)^2/2\epsilon}}\,.
% \left( 
% \frac{\nabla_y \rho^N_y(y_j(t))}{\rho^N_y(y_j(t))} - 
% \frac{\nabla_y \rho^{\ast}_y(y_j(t))}{\rho^{\ast}_y(y_j(t))} \right)\,,
\end{equation}
If the observed data $\rho^{\ast}_y$ is also an empirical distribution, we can apply kernel density estimation to obtain an approximated reference density.

\subsection{Adjoint Solver Simplification}\label{subsec:adjoint}
When $\calG$ is explicitly given, $\nabla_u\calG$ in~\eqref{eqn:particle_KL} is rather immediate. However, in many situations, $\calG$ is generated by the underlying differential equations, and the explicit calculation of $\calG$ relies on PDE solvers. Computing the associated gradient would be even more complicated. In particular, recall $\calG$ maps $\mathbb{R}^m$ to $\mathbb{R}^n$. The gradient is stored in a Jacobian matrix of size $n \times m$. As such, the preparation of the entire matrix directly calls for $mn$ partial derivatives computations, each of which, in turn, calls for a PDE solver. The associated computational cost is prohibitive.

To simplify the computation, we note that in~\eqref{eqn:particle_KL}, the Jacobian matrix is applied to a vector as the source term for updating $u_j(t)$. So instead of preparing for the whole Jacobian matrix and then multiplying it on a vector, one can directly compute the matrix-vector product on the equation level based on the adjoint approach~\cite[Sec.~3.3]{nurbekyan2022efficient}, as summarized below.

For a fixed parameter $u$, instead of dealing with the explicit map $y=\calG(u)$, we consider the following implicit relation between $u$ and $y$ that encodes the PDE information:
\begin{equation}\label{eq:constraint}
    g(y , u) = 0.
\end{equation}
{Here $g$ is the PDE operator that maps $u$ from parameter space and $y$ from PDE solution space to the right hand side of the PDE. We assume $g$ is Fr\'echet differentiable in both arguments.} Based on the first-order variation of~\eqref{eq:constraint} in both $u$ and $y$, we have
\begin{equation*} %\label{eq:IFT}
 \nabla_y g \, \nabla_u y  +\nabla_ug = 0\quad\Rightarrow\quad   \nabla_y g \,\nabla_u\calG \,   =  -\nabla_u g\,.
\end{equation*}
To compute $\nabla_u\calG^\top \xi$, we can set
\begin{equation} \label{eq:adj_eqn}
    \nabla_yg^\top \lambda =  \xi,
\end{equation}
which immediately translates to
\begin{equation}\label{eqn:assemble_gradient}
\nabla_u\calG^\top \xi=-\nabla_ug ^\top \lambda\,.
\end{equation}
We should note that in most PDE settings, $g$ maps the function space of $y$ and $u$ to a function space, and the notation of $\nabla_yg ^\top\lambda=\xi$ really means the inner product taken on the function space of $g$ and its dual, where $\lambda$ is chosen from. This typically translates to the adjoint PDE solver. Such adjoint solution then gets integrated in~\eqref{eqn:assemble_gradient} for the final functional gradient.

An illustrative example is to consider the acoustic wave equation on a spatial domain $\Omega$ and time interval $[0,T]$,
\begin{equation}\label{eq:wave_exp}
    u(x) \partial_{tt} y(x,t)- \Delta y(x,t) = s(x,t), \quad y(x,t=0) = \partial_t y(x,t=0) =0\,,
\end{equation}
where $y(x,t)$ is the wave equation solution, $s(x,t)$ is the given source term, and $u(x)$ is the squared slowness representing the medium property for the wave propagation. Without loss of generality, we consider $\Omega = \mathbb{R}^d$. For a fixed $s$, the solution $y$ is determined by $u$, but in general, we do not have an explicit formulation for the forward map $\calG$ such that $y = \calG(u)$. Instead, the implicit relation~\eqref{eq:wave_exp} is used. We can write~\eqref{eq:wave_exp} as
\[
g(u,y) = 0,\quad \text{where}\quad g(u,y) = u\, \partial_{tt} y - \Delta y -s\,.
\]
To compute the linear action of the adjoint Jacobian, $\xi\mapsto \nabla_u\calG ^\top \xi$, we need to solve the adjoint equation. To this end, we first easily obtain $\nabla_u g = \partial_{tt} y$. To compute $\nabla_yg^\top\lambda=\xi$, we realize that
\[
\nabla_yg^\top \lambda = \nabla_y\langle u\partial_{tt}y-\Delta y-s\,,\lambda\rangle_{x,t} = \nabla_y\langle y\,,\partial_{tt}(u\lambda)-\Delta \lambda\rangle_{x,t} = u \partial_{tt} \lambda-\Delta \lambda\,,
\]
and $\lambda$ satisfies zero final-time condition, i.e., $\lambda(x,t=T) = \partial_t \lambda(x,t=T) = 0$. {Here, $\langle \cdot, \cdot \rangle_{xt}$ represents integration in both $x$ and $t$ domains}, and we performed two levels of integration by parts leading to the equation of
\begin{equation}\label{eqn:adjoint}
u(x)\partial_{tt}\lambda -\Delta \lambda = \xi\,, \quad \text{with zero final condition}\,.
\end{equation}
Assembling the functional derivative according to~\eqref{eqn:assemble_gradient}, we have
\[
\nabla_u\calG^\top\xi=- \int_0^T\int_{\mathbb{R}^d} \partial_{tt} y  (x,t)\lambda(x,t) \rd x\rd t\,.
\]
Note that, as usual, $y$ solves the forward wave equation, and $\lambda$ solves the adjoint equation~\eqref{eqn:adjoint}. For this particular problem, the wave equation PDE operator is self-adjoint, making the forward and adjoint equations have the same form, except for the different source terms and the boundary conditions in time.

Returning to~\eqref{eqn:particle_KL}, to update the values for $u_j$, one can first compute $$\xi_j(t)=\nabla_y\left.\frac{\delta D}{\delta\rho_y}\right|_{\rho^N_y}(y_j(t))\,,$$
and solve the adjoint equation~\eqref{eq:adj_eqn} with respect to $\lambda$ before finally assembling the directional derivative in~\eqref{eqn:assemble_gradient}. The algorithm is summarized in Algorithm~\ref{alg:partical_method}.
\begin{algorithm}[h]
   \caption{Particle method for \eqref{eqn:GF_KL} with $g(u, y) = 0$ where $y = \calG(u)$ is an implicit forward map.}\label{alg:partical_method}
\begin{algorithmic}
\STATE \textbf{Input:} $\rho_y^*(y)$, initial guess $\{u_j^0\}_{j=1}^N$, and step size $\Delta t$.
\FOR{Iteration $n = 0, 1, 2, \cdots, N_{\max}$}
\STATE 1. Set $y_j^n = \calG(u_j^n)$, $j = 1, \cdots, N$
\STATE 2. Compute $\xi_j$ according to~\eqref{eqn:particle_KL}.
\STATE 3. Solve for $\lambda_j$ from the adjoint equation~\eqref{eq:adj_eqn} for every $\xi_j$.
\STATE 4. Update $u_j$ according to~\eqref{eqn:assemble_gradient}.
\ENDFOR
\STATE \textbf{Output:} Final particle locations $\{u^{N_{\max}}_j\}_{j=0}^N$.
\end{algorithmic}
\end{algorithm}

\subsection{Discussion on the Particle Method}
In this subsection, we will examine the similarities and differences between the particle method and other similar systems, highlighting the shared features as well as the distinguishing characteristics.

The first system we compare~\eqref{eqn:GF_u} with is the Langevin dynamics, the continuous version of the classical Langevin Monte Carlo (LMC) algorithm~\cite{dalalyan2017further}
developed to perform Bayesian sampling:
\[
\rd y_t = -\nabla_yf(y_t)\rd{t} +\rd W_t\,,
\]
where $W_t$ denotes the Wiener process. It is a convention to denote by $\pi(t,y)$ the associated probability distribution along time $t$, and apply the It\^{o} calculus to obtain the following Fokker--Planck equation:
\[
\partial_t\pi - \nabla_y\cdot (\nabla_yf\pi+\nabla_y\pi)=0\,,
\]
which can be viewed as the Wasserstein gradient flow of energy $E$
\begin{equation}\label{eqn:GF_LMC}
\partial_t\pi - \nabla_y\cdot\left(\pi\nabla_y\frac{\delta E}{\delta\pi}\right)=0\,,\quad\text{with}\quad E(\pi) = \mathrm{KL}(\pi|\pi^\ast)\,,
\end{equation}
where $\pi^\ast \propto e^{-f}$ is the target distribution. This shows that one interpretation of the Fokker--Planck PDE is that the evolution represents a first-order descending scheme that pushes $\pi$ to the target $\pi^\ast$ in the long time horizon when the distance/divergence and the metric $d_g$ are set to be $\mathrm{KL}$ and $W_2$, respectively. Algorithmically, this means that LMC is the first-order method to draw samples from a target distribution, justifying LMC's validity (asymptotic under some conditions).

While the original LMC requires samples from the Wiener process by adding Gaussian random variables in the updating formula, the corresponding gradient flow equation~\eqref{eqn:GF_LMC} admits a much more straightforward particle method. Namely,  we directly let the particles descent in the negative gradient direction:
\begin{equation}\label{eqn:sampling_GF}
\frac{\rd}{\rd t} y =-\nabla_y\log\left(\pi/\pi^\ast\right)\,,\quad\text{where}\quad y\sim\pi\,.
\end{equation}
As with~\eqref{eqn:GF_u}, this method is not immediately feasible due to the lack of explicit form of $\pi$, so in practice, we set $\pi\sim\pi^N=\frac{1}{N}\sum_i\delta_{y_i}$ as the ensemble distribution, and obtain the following updating formula for the particle method:
\[
\frac{\rd}{\rd t}y_i = -\left(\nabla\log\pi^N(y_i)-\nabla\log\pi^\ast(y_i)\right)\,.
\]
The same issue on the singularity of computing $\nabla\pi^N$ arises, and the blob method was constructed to mitigate such difficulties; see~\cite {carrillo2019blob} for a reference. 

We should note the strong similarity between the above formulation~\eqref{eqn:sampling_GF} and Equation~\eqref{eqn:GF_y}. The main difference between our approach and LMC sampling is that, in our framework, the sampler's motion is presented in the $u$ domain, and when translated to $y$-domain, must be projected onto the space spanned by the $\nabla_u\calG \,\nabla_u\calG^\top$, whereas LMC sampling operates directly on $y$ without projection. In some sense, the new formulation can be viewed as a projected gradient flow onto the tangent kernel space with the kernel spanned by columns of $\nabla_u\calG$.

On the other hand, the connection to LMC also inspires the possibility of introducing stochasticity, particularly the Brownian motion, to avoid dealing with Dirac delta functions. Indeed, denoting $C(u(t)) = \left.\nabla_u\calG\right|_{u(t)} \, \left.\nabla_u\calG\right|^\top_{u(t)}$, and assuming that $\calG$ is invertible, we set $B(y) = C(\calG^{-1}(y))$. Then the evolution equation for $\rho_y$ is explicit from~\eqref{eqn:GF_y} or \eqref{eqn:GF_KL}:
\begin{equation*}\label{eqn:FP_y}
\partial_t\rho_y = \nabla_y\cdot\left(\rho_y\,B(y)\,\nabla_y\log\left(\frac{\rho_y}{\rho_y^\ast}\right)\right) = \nabla_y\cdot\left(B\,\nabla_y\rho_y+\rho_y\,B\,\nabla_yf\right)\,,
\end{equation*}
where we assume $\rho_y^\ast\propto e^{-f(y)}$. As a consequence, we have the following stochastic particle method: 
\begin{equation}\label{eqn:LMC_y}
\rd y_t = -B(y(t))\, \nabla_yf(y_t)\rd{t} + \sqrt{2B(y(t))}\rd{W}_t\,.
\end{equation}
In the most simplified case, consider a linear dependence by letting $y = \calG(u) = \mathsf{A}u$. Then we have $\nabla_u\calG=\mathsf{A}$, and~\eqref{eqn:LMC_y} becomes:
\begin{equation*}
\rd y_t =  - \mathsf{A}\mathsf{A}^\top\,\nabla_yf\rd{t}+\sqrt{2\mathsf{A}\mathsf{A}^\top}\rd{W_t} \,,
% \nabla_y\log\left(\rho_y/\rho_y^\ast\right)\,,\quad\text{where}\quad y\sim\rho_y\,.
\end{equation*}
a formulation that resembles Ensemble Kalman Sampler developed in~\cite{garbuno2020interacting}, where the matrix in front of $\nabla_y f$ is the data ensemble covariance matrix.

Finally, we draw the connection to the mirror descent method that the update formula \eqref{eqn:particle_KL} for $u$ carries~\cite{beck2003mirror, li2022mirror, wang2022hessian}. The mirror descent performs gradient descent on the mirror variable, defined by taking the gradient of a convex function. For example, one can define a convex function $\phi$ on $\mathbb{R}^m\ni x$ and the mirror variable $z(x) = \nabla_x\phi(x)$. The mirror descent in the continuous-in-time limit represents:
\[
\frac{\rd}{\rd t} z = - \nabla f(x)\,,\quad\text{or equivalently}\quad
\frac{\rd}{\rd t} x = -H_\phi^{-1}\nabla f(x)\,,
\]
where $H_\phi(x)$ is the Hessian term of $\phi$ evaluated at $x$. In our case shown in~\eqref{eqn:GF_y}, if we view $y$ as $x$ and $u$ as $z$, the descending formula in~\eqref{eqn:GF_y} writes:
\begin{equation*} %\label{eq:ours2}
    \frac{\rd}{\rd t} x = -\nabla_z \calG \big |_{z(t)} \, \nabla_z \calG^\top \big |_{z(t)}  \, \nabla_x f (x)\,.
\end{equation*}

If $\nabla_z \calG(z)$ is of full row rank for any fixed $z$, the matrix $\nabla_z \calG  \nabla_z \calG^\top$ is naturally strictly positive definite. If one can view it as the hessian of a convex function $\phi$, meaning $H_\phi^{-1}(z) = \nabla_z \calG \nabla_z \calG^\top$, 
then our update formula~\eqref{eqn:particle_KL} can be seen as a mirror descending procedure on $u$ using the mirror function $\phi$.

\section{Well-posedness Result for Linear Push-Forward Operators}\label{sec:theory}
Like deterministic PDE-constrained optimization problems, when the differential equation has parameters that have inherent randomness, the above formulated SDE-constrained optimization using the gradient-flow structure and the associated particle method for implementing the gradient flow are rather generic: the dimensions $m$ and $n$ in \eqref{eqn:G} can be arbitrary in the execution of the algorithm. However, the performance of the gradient descent algorithm and its capability of capturing the global minimum highly depends on the structure of the forward map $\mathcal{G}$. This makes providing a full-fledged convergence theory impossible.

Nevertheless, we can pinpoint certain properties when the push-forward map $\calG$ is linear. Resonating the situation in the deterministic setting, we carry out the studies by separating the discussion into over-determined and under-determined scenarios. In each scenario, we will begin our discussion with the deterministic setup \eqref{eqn:argmin_pde} with respect to the $L^2$ geometry and proceed with the stochastic setup~\eqref{eqn:argmin_sde}, emphasizing the similarity.

First, we fix the notation. Since $\mathcal{G}$ {is} linear, we denote it as
\[
y = \mathcal{G}(u)=\mathsf{A}u\,,\quad \text{with}\quad \mathsf{A}\in\mathbb{R}^{n\times m}\,
\]
throughout this section. 
We also assume that  $\mathsf{A}$ is full-rank, in the sense that $\text{rank}(\mathsf{A})=\min\{m,n\}=:r$, so there are no redundant rows/columns, and the size of the matrix determines if the system is under or over-determined. 
Additionally, we conduct the compact SVD for $\mathsf{A}$ and write
\begin{equation} \label{SVD}
 %\mathsf{A}^\top = \mathsf{U}\mathsf{S}\mathsf{V}^\top\,,\quad\Rightarrow\quad
 \mathsf{A}=\mathsf{V}\mathsf{S}\mathsf{U}^\top\,,\quad\text{with}\quad 
 \mathsf{V}\in\mathbb{R}^{n\times r}\,, \quad  \mathsf{S}\in\mathbb{R}^{r\times r}\,, \quad 
 \mathsf{U}\in\mathbb{R}^{m\times r}\,.
\end{equation}

To begin with, we realize that in this linear setting, the gradient flow equation~\eqref{eqn:GF} is much more explicit:
\begin{align}\label{eqn:GF_linear}
    \partial_t\rho_u - \nabla_u\cdot\left(\rho_u\nabla_u\left(\frac{\delta E}{\delta \rho_u}\right)\right)=0\,, \quad E(\rho_u) = D(\rho_y, \rho_y^*) =  D(\mathsf{A}_\sharp \rho_u, \rho_y^*)\,.
\end{align}

All discussions below extend this analysis to both under-determined and over-determined scenarios.  We use $\rho$ to denote probability density and measure interchangeably.

\subsection{Under-Determined Scenario}\label{sec:under_determine}
Under-determined scenario corresponds to the case when $n\leq m$. That is, the number of the to-be-determined parameters is no fewer than the collected data, and matrix $\mathsf{A}$ is short-wide. We incorporate the fully determined matrix ($n=m$) in this regime as a special case. When $n <m$, in either deterministic or stochastic settings, we expect infinitely many solutions to exist. 

\subsubsection{Deterministic Case}
The optimization~\eqref{eqn:argmin_pde} in the linear case becomes
\begin{equation}\label{eqn:under_determine_deter}
\min_u\|\mathsf{A}u -y^*\|^2\,.
\end{equation} 
When $\Amat$ is fully determined, the solution is unique, but when $n<m$, there are infinitely many possibilities to choose $u$ so that the $y^*=\mathsf{A}u $ exactly. In practice, to select a unique parameter, one typical approach is to add a regularization term. This way, the selected parameter configuration not only minimizes~\eqref{eqn:under_determine_deter} but also satisfies certain properties known a priori. One of the most classic examples is Tikhonov regularization: $\min_u\|\mathsf{A}u - y^*\|^2 +\|u-u_0\|^2$. This ensures that the error term $y-\mathsf{A}u$ is small and that the optimizer is relatively close to the suspected ground truth $u_0$. Here we will take a different point of view. Instead of adding a regularization, we run gradient descent on the vanilla objective~\eqref{eqn:under_determine_deter}. As expected, the choice of the initial guess $u_\text{i}$
plays the role of selecting a unique minimizer. In other words, the optimization method itself, which, in our case, is the gradient descent algorithm, has an implicit regularization effect on the converging solution.

To this end,  we first augment $\mathsf{A}$ with $\tilde{\mathsf{A}}$ to form a rank-$m$ matrix, and correspondingly define the augmented $y^\ex$:
\begin{equation}\label{eqn:def_A_ex}
\mathsf{A}^\text{ex}=\left[\begin{array}{c}{\mathsf{A}} \\ \tilde{\mathsf{A}}\end{array}\right]\,,\quad y^\ex = \mathsf{A}^\ex u =\left[\begin{array}{c}{\mathsf{A}} u \\ \tilde{\mathsf{A}} u\end{array}\right]=\left[\begin{array}{c}y \\ \tilde{y}\end{array}\right]\,.
\end{equation}
{Note the definition of $\tilde{\mathsf{A}}$ is not unique but it does not affect the upcoming calculation. The easiest choice is to set $\tilde{\mathsf{A}}^\top=\mathsf{U}^\perp$, where $\mathsf{U}^\perp$ is the orthogonal complement of $\mathsf{U}$ given in~\eqref{SVD}.} As a result, the right singular vector set for $\tilde{\mathsf{A}}$ is $\mathsf{U}^\perp$. 
 Suppose $u^\ast \in \{u : \mathsf{A} u = y^*\}$. Then the solution set can be written as
\begin{equation}\label{eqn:equi_set_determ}
\mathcal{S}=\{u^\ast+\tilde{u}:\quad\mathsf{A}\tilde{u}=0\}=\{u^\ast+\text{span}\mathsf{U}^\perp\}\,.
\end{equation}

The particular solution chosen from the solution set $\mathcal{S}$ is uniquely determined by the optimization process and the initial data. Suppose we perform gradient descent on~\eqref{eqn:under_determine_deter}. Then in the continuous-time limit, it amounts to solving the following ODE:
\begin{equation}\label{eqn:theta_ODE_determ}
\frac{\rd u }{\rd t}= - \mathsf{A}^\top\left(\mathsf{A}u-y^*\right)\,,\quad\text{with}\quad u(t=0)=u_\text{i}\,.
\end{equation}
The following result shows that the gradient descent leads us to the solution that agrees with $u_0$ when projected onto $\mathsf{U}^\perp$ and agrees with $u^\ast$ when projected onto $\mathsf{U}$.
\begin{proposition}\label{prop:under_determine}
The {equilibrium} solution to~\eqref{eqn:theta_ODE_determ}, denoted by $u_\text{f}$, given the initial iterate $u_\text{i}$, can be written as
\[
u_\text{f}=\mathsf{U}\mathsf{U}^\top u^\ast+ \mathsf{U}^\perp(\mathsf{U}^\perp)^\top u_\text{i}\,.
\]
Moreover, $y^\ex_\text{f} = \mathsf{A}^\ex u_\text{f}$ as defined in~\eqref{eqn:def_A_ex} satisfies:
\begin{equation}\label{eqn:y_proj_under}
y_\text{f} = y^\ast\,,\quad\text{and}\quad \tilde{y}_\text{f} = \tilde{\mathsf{A}} u_\text{i}\,.
\end{equation}
\end{proposition}

One aspect of this result is that the generated solution, when confined to the space spanned by $\mathsf{A}$, entirely agrees with the given data $u^*$, and when confined to the augmented section, purely agrees with the generated data from the initial guess $u_0$. When $n=m$, $\mathsf{U}^\perp$ only contains the zero vector, so $u_\text{f}\equiv u^*$. The proof of the proposition is rather standard and put in~\Cref{app1} for completeness.

\subsubsection{Stochastic Case}
The situation in the stochastic setting is an analogy to that in the deterministic setting. In the current under-determined situation, the push-forward map $\calG = \mathsf{A}$ has more degrees of freedom to pin in the parameter space than the data offers. Therefore, it is guaranteed that one can find infinitely many $\rho_u$ that achieves the exact match:
\begin{equation*}\label{eqn:exact_match_stoch}
\rho_y = \mathsf{A}_{\sharp}\rho_u = \rho_y^\ast\,.
\end{equation*}

As with the deterministic case, the particular solution in the solution set $\mathcal{S}$ we obtain is determined by the initial guess and the optimization algorithm in a combined manner.
We expect the same to hold in the stochastic case. Furthermore, we use the same notation as in~\eqref{eqn:def_A_ex}. The final conclusion is an analogy of Proposition~\ref{prop:under_determine}.
\begin{theorem}\label{thm:under_determine_stoch}
{Suppose~\eqref{eqn:GF_linear} using the initial data $\rho_u^0$ has an equilibrium solution, and we denote it to be $\rho_u^{\infty}$}, and let $\rho_{y^\text{ex}}^{\infty}$ be the push forward density of $\rho_u^{\infty}$ under the map $\Amat^{\text{ex}}$, i.e., 
    \begin{align} \label{0215}
       \rho_{y^\ex}^\infty =  \Amat^{\text{ex}}_\sharp \rho_u^\infty \,.
    \end{align}
    Then we can uniquely determine the marginal {distributions} of $\rho_{y^\ex}^{\infty}$, in the sense that
    \begin{itemize}
    \item The marginal distribution on $y$ of $\rho_{y^\ex}^{\infty}$ entirely recovers that of the data $\rho_y^*$,
   \item The marginal distribution on $\tilde{y}$ of $\rho_{y^\ex}^{\infty}$ is uniquely determined by that of $\rho_y^{0}$.
\end{itemize}
\end{theorem}
A very natural corollary of~\Cref{thm:under_determine_stoch}, when $\Amat$ is fully determined, suggests the unique recovery of the ground truth.
\begin{corollary} 
When $n=m$, i.e., the matrix $\Amat$ is fully determined, under the same assumptions of~\Cref{thm:under_determine_stoch}, the equilibrium solution
\[
       \rho_y^\infty =  \Amat_\sharp \rho_u^\infty  = \rho_y^*\,,
\]
is unique and independent of the initial distribution.
\end{corollary}

To prove~\Cref{thm:under_determine_stoch}, we realize that since the system is under-determined, there are infinitely many solutions. The particular solution we get is the single distribution function that falls at the intersection of the solution set, denoted by $\mathcal{S}_\rho$ and defined in~\eqref{eqn:soln_set_GF} and the gradient flow dynamics~\eqref{eqn:GF_linear}. The proof is similar to what we had in the deterministic setting: we first identify the class of functions on the gradient flow dynamics and then evaluate when and how they intersect with $\mathcal{S}_\rho$.

The following lemma first characterizes its solution by following the flow trajectory.  
\begin{lemma}\label{lem:1}
Suppose one starts the gradient flow equation~\eqref{eqn:GF_linear} with the initial condition $\rho_u(u,0) = \rho^0_u$. Then the solution lives in the following set:
\begin{equation}\label{eqn:rho_stoch_trajectory}
\rho_u(t)\in\{\rho_u^0+h\,,\quad\text{with}\quad\int h(u) \rd u =0\,,\quad\tilde{\mathsf{A}}_{\sharp}h=0\}\,,
\end{equation}
where $\tilde{\mathsf{A}}$ is defined in \eqref{eqn:def_A_ex}.
\end{lemma}
\begin{proof}
Let $\rho_u$ be the solution to~\eqref{eqn:GF_linear} and $\rho_u^0$ be the initial distribution. Then $\rho_u(u,\tau)$ differs from $\rho_u^0$ by
\[
\rho_u(u,\tau ) - \rho_u^0(u) =h(u;\tau):= \int_0^\tau \nabla_u\cdot\left(\rho_u(t) \nabla_u\frac{\delta D}{\delta\rho_y(t)} (\mathsf{A} u)\right) \rd t\,,
\]
a function of $u$ parameterized by $\tau$. Essentially, to demonstrate that $\int h(u)\rd{u}=0$ and $\tilde{\mathsf{A}}_\sharp h = 0$, we only need to establish the validity of these two equations for any function $q(u)$ in the form of
\[
q(u)=\nabla_u\cdot\left(\rho_u\nabla_u\frac{\delta D}{\delta\rho_y} (\mathsf{A} u)\right)\,.
\]
The mean-zero property $\int q(u)\rd u=0$ is easy to show given that $q$ has a divergence form. To show $\tilde{\mathsf{A}}_{\sharp}q=0$, we recall that for any $\psi$ that is integrable with respect to $q$ upon composed with $\tilde{\mathsf{A}}$, by performing integration by parts, we have:
    \begin{align*}
    \int\psi(\tilde{\mathsf{A}}u )q(u)\rd u 
    %& = -\int\nabla_u \psi(\tilde{\mathsf{A}} u ) \cdot \nabla_u\frac{\delta D}{\delta\rho_u}\rho_u ( u)\rd u \\ 
    & =  -\int\nabla_u \psi(\tilde{\mathsf{A}} u ) \cdot \nabla_u\left( \frac{\delta D}{\delta \rho_y}(\mathsf{A}u) \right) \rho_u( u)\rd u \\ 
    & = -\int  \underbrace{\tilde{\mathsf{A}}^\top \nabla_y \psi( y ) \big|_{y = \tilde{\mathsf{A}} u}}_{\xi_1}  \cdot  \underbrace{\mathsf{A}^\top \nabla_y\left( \frac{\delta D}{\delta \rho_y}(y) \right) \big|_{y = \mathsf{A} u}}_{\xi_2} \, \rho_u( u) \rd u \,.
    \end{align*}
Note that the row space of $\mathsf{A}$ and $\tilde{\mathsf{A}}$ are $\mathsf{U}$ and $\mathsf{U}^\perp$, respectively, so $\xi_{1}$ and $\xi_2$ belong to these two perpendicular spaces. We then have $\int\psi(\tilde{\mathsf{A}}u)q(u)\rd u=0$, yielding $\tilde{\mathsf{A}}_{\sharp}q=0$.
Since at each infinitesimal time, $q$ satisfies these two conditions, we prove~\eqref{eqn:rho_stoch_trajectory}.
\end{proof}

Next, we study the steady state of the gradient flow system~\eqref{eqn:soln_set_GF}:
\begin{lemma} \label{lem2}
When $D$ is the KL divergence, and let $\rho_u^\ast$ be a reference probability measure (i.e., $\Amat_\sharp \rho_u^* = \rho_y^*$), then all equilibrium of \eqref{eqn:GF_linear} are in the following set: 
\begin{equation}\label{eqn:soln_set_GF}
    \mathcal{S}_\rho=\{\rho_u^\ast+g:\,\int g(u)\rd u=0\,,\mathsf{A}_{\sharp}g=0\}\,.
\end{equation}
Here, $\mathsf{A}_\sharp g=0$ can also be interpreted as
$\int\psi(\mathsf{A}u)g(u)\rd u = 0$, which means that for all $\psi$, when composed with $\mathsf{A}$, is $g$-integrable to 0.
\end{lemma}
\begin{proof}
Recall~\eqref{eqn:GF_linear} and the discussion leading to~\eqref{0312}, we have that the equilibrium set consists of the states $\rho^\infty_u$ at which the Fr\'echet derivative is trivial on the support. This, combined with~\eqref{eqn:Frechet}, gives:
\begin{equation*} \label{0313}
   \nabla_u \left.\frac{\delta E}{\delta \rho_u}\right|_{\rho_u^\infty}(u) =  \nabla_u \left.\frac{\delta D}{\delta \rho_y}\right|_{\mathsf{A}\sharp\rho_u^\infty} (\Amat u)  = \left.\Amat^\top \nabla_y \frac{\delta D}{\delta \rho_y}\right|_{\mathsf{A}\sharp\rho_u^\infty} (\Amat u) = 0 \,.
\end{equation*}
Since $\Amat$ is a flat matrix with full rank,
\begin{align} \label{0407}
\left.\Amat^\top \nabla_y \frac{\delta D}{\delta \rho_y}\right|_{\mathsf{A}\sharp\rho_u^\infty} (\Amat u) = 0\,,\quad \text{on the support of}~ \rho_u^\infty \,.
\end{align}
In the case when $D$ is the KL divergence, based on \eqref{eqn:KL_obj}, we have that  $\rho_y^\infty \propto \rho_y^*$. Since they both are probability measures, we can conclude that $\rho_y^\infty = \rho_y^*$. Since $\rho_y^\infty = \Amat_\sharp \rho_u^\infty$, this means that $\rho_u^\infty$ is in $\mathcal S_\rho$ defined in~\eqref{eqn:soln_set_GF}. 
\end{proof}

\begin{remark}
    To show that $\rho_y^\infty = \rho_y^*$ from \eqref{0407}, $D$ only needs to be strictly displacement convex with respect to $\rho_y$ (see Section 5.2.1 in \cite{villani2021topics} for its definition). Indeed, note that if we view $D(\rho_y^\infty, \rho_y^*)$ as a functional of $\rho_y^\infty$, the equilibrium set of its Wasserstein gradient flow is \eqref{0407}. Then under the convexity condition mentioned above, the equilibrium set has one element \cite{mccann1997convexity}, which is the unique minimizer of $D(\rho_y^\infty, \rho_y^*)$.
\end{remark}

These two lemmas prepare us to investigate the specific solution by following~\eqref{eqn:GF_linear} from an initial $\rho_u^0$. We are ready to prove the main theorem.
\begin{proof}[Proof for Theorem~\ref{thm:under_determine_stoch}]
From~\eqref{eqn:rho_stoch_trajectory}-\eqref{eqn:soln_set_GF} we see that the equilibrium $\rho_u^\infty$ admits the following expression 
    \begin{align*} 
        \rho_u^\infty = \rho_u^0 + h = \rho_u^* + g\,,  
    \end{align*}
    where $\int h \rd u = \int g \rd u = 0$ and $\Amat_\sharp g = 0$, $\tilde \Amat_\sharp h = 0$. Therefore 
    \[
    \Amat_\sharp \rho_u^\infty = \Amat_\sharp \rho_u^*\,, \quad
    \tilde \Amat_\sharp \rho_u^\infty =  \tilde \Amat_\sharp \rho_u^0\,.
    \]
   That is, for any integrable test function $\psi $% under appropriate sense \yy{what sense?}
   \begin{equation}\label{0217}
  \int \psi(\Amat u) \rho_u^\infty (u) \rd u = \int \psi(\Amat u) \rho_u^* (u) \rd u \,,\quad \int \psi(\tilde \Amat u) \rho_u^\infty (u) \rd u = \int \psi(\tilde \Amat u) \rho_u^0 (u) \rd u \,.
   \end{equation}
   From \eqref{0215}, we have that $\rho_{y^\ex}^\infty (\Amat^{\text{ex}} u ) \det{(\Amat^{\text{ex}})} = \rho_u^\infty (u)$. Using the notation in \eqref{eqn:def_A_ex}, the two equations above become
\begin{align*}
  &  \int \psi(y) \rho_y^\infty (y) \rd y =\underbrace{\int \psi(y) \rho_{y^\ex}^\infty (y^\ex) \rd y^\ex = \int \psi(y) \rho_{y^\ex}^\ast (y^\ex) \rd y^\ex}_{\text{from}~\eqref{0217}}= \int \psi(y) \rho_y^* (y) \rd y\,,% \quad \text{where}~ \rho_y^* =  \Amat^{\text{ex}}_\sharp \rho_u^* \,,
\end{align*}
where we used $\rho_{y^\ex}^* =  \Amat^{\text{ex}}_\sharp \rho_u^*$ and the definition of the marginal distributions $\rho^{*}_y(y) = \int\rho^{*}_{y^\ex}(y^\ex=[y,\tilde{y}])\rd\tilde{y}$, and $\rho^{\infty}_y(y) = \int\rho^{\infty}_{y^\ex}(y^\ex=[y,\tilde{y}])\rd\tilde{y}$ for the first and the last equal signs. Similarly the second equation in~\eqref{0217} becomes
\begin{align*}
  \\ &  \int \psi(\tilde{y}) \rho_{y^\ex}^\infty ({y^\ex}) \rd y^\ex = \int \psi(\tilde{y}) \rho_{y^\ex}^0 ({y^\ex}) \rd y^\ex \,,\quad \text{where}~ \rho_y^0 =  \Amat^{\text{ex}}_\sharp \rho_u^0\,.
   \end{align*}
Considering that $\psi$ can be chosen as any function, we obtain all the moments, and thus the marginal distribution $\rho_{y^\ex}^\infty$ on $y$ and $\tilde{y}$, finishing the proof.
\end{proof}

\subsection{Over-Determined Scenario}
Over-determined scenario corresponds to the case when $n>m$, so there are more data to fit than the to-be-determined parameters. It is often unlikely to find a solution that satisfies the equation exactly, but one can nevertheless find the best approximation that minimizes a chosen misfit function. 
\subsubsection{Deterministic Case}
As an example, we look for the configuration that provides the minimum misfit under the vector 2-norm. That is,
\[
\min_u\frac{1}{2}\|\mathsf{A}u - y^*\|_2^2\,.
%=u^\top\mathsf{A}^\top\mathsf{A}u-2{y^*}^\top\mathsf{A}u+{y^*}^\top y^*\,.
\]
For a linear system like this, the minimizer is explicit:
\begin{equation}\label{eq:det_over}
u^\ast=(\mathsf{A}^\top\mathsf{A})^{-1}\mathsf{A}^\top y^*=:\mathsf{A}^\dagger y^*\,,
\end{equation}

and hence, calling~\eqref{SVD}
\begin{equation}\label{label:opt_mu_determ}
y = \mathsf{A}u^\ast = \mathsf{A}\mathsf{A}^\dagger y^*= \mathsf{V}\mathsf{V}^\top y^*\,,\quad\text{or equivalently}\quad y = y^\ast_{\mathsf{A}}=\text{Proj}_\mathsf{V}\,  y^*\,.
\end{equation}
In other words, $y$ is $y^*$ projected onto the column space of $\mathsf{V}$, (also the column space of $\Amat$), and thus $\mathsf{V}^\top y = \mathsf{V}^\top y^*$ and $\left(\mathsf{V}^\perp\right)^\top y=0$.

\subsubsection{Stochastic Case}
In the stochastic case, we expect to obtain a similar result that suggests $\rho_y^\infty$ is a  ``projection'' of $\rho_y^\ast$ onto the column space of $\mathsf{A}$. More specifically, we have the following theorem that characterizes the equilibrium $\rho_y^\infty $ when $\rho_y^*$ has a separable structure.

\begin{theorem}\label{thm:over_determin_stoc}
 Let $\rho_u^{\infty} $ be the equilibrium solution to~\eqref{eqn:GF_linear}. Consider $D$ to be KL, i.e., $ D(\rho_y, \rho_y^*) = \int \rho_y \log \frac{\rho_y}{\rho_y^*} \rd y$,  and $\rho_y^*$ has the separable form: 
\begin{equation*}\label{0315}
 \rho_y^* (y)= e^{f_1(y_\Amat)}e^{ f_2(y_{\Amat^\perp})}
\end{equation*}
where $y_\Amat =\Vmat\Vmat^\top y$, {with $\Vmat$ being the column space of $\Amat$,} and $y_{\Amat^\perp}=y-y_\Amat$ being the perpendicular component. Then 
\begin{equation*}\label{0316}
\rho_y^\infty(y) = \Amat\Amat^\dagger_\sharp\rho_y^\ast\propto  e^{f_1(y_\Amat)} \,.
\end{equation*}
\end{theorem}
This Theorem is the counterpart of its deterministic version stated in~\eqref{label:opt_mu_determ}, showing that we can identify partial information of $\rho_y^\infty$---the part that is in the column space of $\Amat$ (or equivalently $\Vmat$). The perpendicular direction information is projected out.

\begin{proof}
To start off, we assume $\rho_y^\infty(y)\propto e^{g(y_\Amat)}$, and we will show that $g$ is the same as $f_1$ up to a constant addition. It is safe to assume $\rho_y^\infty(y)$ has no $y_{\Amat^\perp}$ component because it is pushed forward by $\Amat$, so its support is confined to the range of $\Amat$. 

Recall the argument leading to~\eqref{0312}, and combine it with the calculation~\eqref{eqn:Frechet}. We have that the equilibrium of \eqref{eqn:GF_linear} in this case satisfies
\[
      \nabla_u \left.\frac{\delta E}{\delta \rho_u}\right|_{\rho_u^\infty} =  \nabla_u \left.\frac{\delta D}{\delta \rho_y}\right|_{\rho_y^\infty} (\Amat u)  =\Amat^\top \nabla_y \frac{\delta D}{\delta \rho_y} \bigg|_{\rho_y^\infty} (\Amat u) = 0\,\,\Rightarrow\,\, \nabla_y \frac{\delta D}{\delta \rho_y}\bigg|_{\rho_y^\infty}(\Amat u) \in\mathsf{A}^\perp\,.
   \]
Considering $\nabla_y = \Vmat\Vmat^\top\nabla_{y_\Amat}+ \Vmat^\perp(\Vmat^\perp)^\top\nabla_{y_{\Amat^\perp}}$, we should have, for all $u$:
\[
 \nabla_{y_\Amat}\frac{\delta D}{\delta \rho_y}\bigg|_{\rho_y^\infty}(\Amat u)  =0\,.
\]
Noting $\nabla_{y_\Amat}\frac{\delta D}{\delta \rho_y}\bigg|_{\rho_y^\infty}  =\nabla_{y_\Amat}\log \rho_y^\infty - \nabla_{y_\Amat}\log\rho_y^\ast$, we then have
\[
\nabla_{y_\Amat}g = \nabla_{y_\Amat}f_1\,,\quad\text{on}\quad\text{Range}(\Amat)\,.
\]
Thus, $g=f_1$ up to a constant, finishing the proof.

\end{proof}

\subsection{Setting Data Discrepancy $D$ to be $W_2$}
Most results developed in this paper so far sets $D$ as the KL divergence. We switch gear and adopt~\eqref{eqn:discrepancy_w2} in this subsection, by setting $D$ to be the quadratic Wasserstein metric $W_2$. In this case, the variational formulation~\eqref{eqn:argmin_sde} becomes
\begin{equation}\label{eq:w2_admin_sde}
\min_{\rho_u\in \mathcal{P}(\mathcal{A})}\; W_2(\mathcal{G}_{\sharp}\rho_u, \rho^\ast_y)\,,
\end{equation}
with $\mathcal{P}(\mathcal{A})$ denoting all probability distributions over the parameter domain $\mathcal{A}$ as usual.

In the fully determined case where $\mathsf{A}$ is invertible, and the minimizer will be uniquely given by ${\mathsf{A}^{-1}}_\sharp \rho^\ast_y$. In the under-determined case, the same argument in Section~\ref{sec:under_determine} still holds and no uniqueness can be obtained. However, when $\mathsf{A}$ is over-determined with full column rank and $\rho_y^*$ is absolutely continuous with respect to the Lebesgue measure, we claim that the unique minimizer to~\eqref{eq:w2_admin_sde} is ${\mathsf{A}^\dagger}_\sharp \rho_y^*$.

\begin{theorem}\label{thm:over_determin_stoc_W}
Consider $D$ to be $W_2$, $\mathcal{A} = \mathbb{R}^m$, and set $\mathcal{G}=\Amat$, an over-determined matrix with full column rank. Moreover, assume the reference data distribution  $\rho_y^*$ is absolutely continuous with respect to the Lebesgue measure, and supported on $\mathcal{R}$, the closure of a bounded connected open set. Then we have that:
\begin{itemize}
    \item The variational problem~\eqref{eq:w2_admin_sde} has a unique minimizer $\mathsf{A}^\dagger_\sharp \rho_y^*$.
    \item The Wasserstein gradient flow of~\eqref{eqn:argmin_sde3}  with $D$ being the $W_2$ metric, as given in~\eqref{eq:GF_W2}, has a unique equilibrium, which is also the unique minimizer $\rho_u^\infty  = \mathsf{A}^\dagger_\sharp \rho_y^*$.
\end{itemize}

\end{theorem}

\begin{proof}
First, we show that $\mathsf{A}^\dagger_\sharp \rho_y^*$ is the global minimizer. Define a set $\mathcal{S}=\{\rho:\rho = \mathsf{A}_\sharp \rho_u\; \text{for}\;\rho_u \in \mathcal{P}(\mathcal{A})\}$ collecting all probability measures pushed by $\mathsf{A}$. %Here, $\mathcal{P}(\mathcal{A})$ denotes all probability measures defined over the domain $\mathcal{A}$. 
Then for any $\rho\in\mathcal{S}$, we have $\text{supp}(\rho)\subset \text{col}(\mathsf{A})$.  Furthermore,
\begin{eqnarray}
    W_2^2(\rho, \rho_y^*) &=& \int_y |T_\rho(y) - y|^2 \rho_y^*(y) \rd y \label{eq:1}\\
    &\geq& \int_y |\mathsf{A}\mathsf{A}^\dagger y - y|^2 \rho_y^*(y) \rd y \label{eq:2}\\
    &\geq & W_2^2\left((\mathsf{A}\mathsf{A}^\dagger)_\sharp\rho_y^*,\; \rho_y^*\right)\,.\label{eq:3}
\end{eqnarray}
In the derivation,~\eqref{eq:1} holds by definition where we denote by $T_\rho$ the optimal transport map from $\rho_y^*$ to $\rho$. The existence of the optimal map is guaranteed since  $\rho_y^*$ is absolutely continuous~\cite[Thm.~1.17]{santambrogio2015optimal}. Inequality~\eqref{eq:2} holds because of the pointwise inequality:
\[
\argmin\limits_{x\in\text{col}(\mathsf{A})} |x - y|^2 = \mathsf{A}\mathsf{A}^\dagger y\quad\Rightarrow\quad|\mathsf{A}\mathsf{A}^\dagger y - y|^2 \leq |T_\rho(y) - y|^2\,,
\]
where we deployed the fact that $\text{supp}(\rho)\subset \text{col}(\mathsf{A})$ and thus $\text{range}(T_\rho) \subset \text{col}(\mathsf{A})$.
Inequality~\eqref{eq:3} comes from the definition of the Wasserstein distance (note that we do not necessarily claim $\mathsf{A}\mathsf{A}^\dagger$ is the optimal map from $\rho_y^*$ to $\left(\mathsf{A}\mathsf{A}^\dagger\right)_\sharp\rho_y^*$). Consequent to this derivation, we have:
\[
\mathsf{A}_\sharp \left(\mathsf{A}^\dagger_\sharp\rho_y^*\right) = \left(\mathsf{A} \mathsf{A}^\dagger\right) _\sharp\rho_y^* =\argmin_{\rho = \mathsf{A}_\sharp \rho_u,\, \forall\rho_u\in\mathcal{P}(\mathcal{A})} W_2(\rho,\rho_y^*)\,,
\]
or equivalently
\begin{equation}\label{eq:W2_over}
\mathsf{A}^\dagger_\sharp\rho_y = \argmin_{\rho_u \in\mathcal{P}(\mathcal{A})} W_2(\mathsf{A}_\sharp\rho_u,\rho_y^*)\,,
\end{equation}
finishing the proof that $\mathsf{A}^\dagger_\sharp \rho_y^*$ is the global minimizer.

Next, we show that $\mathsf{A}^\dagger_\sharp \rho_y^*$ is also the unique equilibrium of the corresponding Wasserstein gradient flow equation~\eqref{eq:GF_W2}. That is, the equilibrium of the gradient flow equation is the global minimizer of~\eqref{eqn:argmin_sde3}. At equilibrium, we have
\[
\nabla_u \cdot \left( \rho_u^\infty  \nabla_u \phi^*( \mathsf{A}u  ) \right) = 0, \quad \text{on the support of}~ \rho_u^\infty\,,
\]
where $\phi^*$ is the Kantorovich potential associated with $\rho_y^\infty = \mathsf{A}_\sharp\rho_u^\infty$; see~\eqref{eq:W2_dual}-\eqref{eq:W2_dual_Frechet}. This implies that 
\begin{equation}\label{eq:phi_eq2}
\Amat^\top \nabla_y \phi^*( y  ) = 0 \,, \quad  \text{on the support of}~ \rho_y^\infty  \,.
\end{equation}
Based on~\eqref{eq:phi_eq2},  we deduce that $$
\nabla_y \phi^*: \text{supp}(\rho_y^\infty)\subset \text{col}(\Amat) \longrightarrow  \text{col}(\Amat)^\perp,
$$
where  $\text{col}(\Amat)^\perp$ denotes the orthogonal complement of the linear  subspace $\text{col}(\Amat)$. 

 On the other hand, one can express the Kantorovich potential with respect to $\rho_y^*$, denoted by $\psi^*$, using the c-transform:
 \[
 \psi^*(x) =  (\phi^*(y))^c :=  \min_{y\in \text{col}(\Amat)}    \left\{ \frac{1}{2}\|x-y\|^2 - \phi^*(y)\right\}\,, \quad x\in \mathcal{R}\,.
\]
For any $x \in \mathcal{R}$, we denote by $y_x$ the minimizer of the above c-transform.  Since the minimization problem is strictly convex, the first-order optimality condition is also sufficient:
\begin{equation}\label{eq:opt_cond}
    x - y_x - \nabla_y \phi^*(y_x) \perp \text{col}(\Amat)\,.
\end{equation}
Based on~\eqref{eq:phi_eq2}, we know $\nabla_y \phi^*(y_x)  \in \text{col}(\Amat)^\perp$. Moreover, $y_x \in \text{col}(\Amat)$. Therefore, \eqref{eq:opt_cond} is equivalent to the following orthogonal decomposition of $x$:
\[
x = y_x + z,\quad z\in \text{col}(\Amat)^\perp\,.
\]
Since orthogonal decomposition with respect to $\text{col}(\Amat)$ is unique, we obtain that $y_x  = \Amat\Amat^\dagger x$, $\forall x \in \mathcal{R}$. Based on the optimal transportation theory~\cite{santambrogio2015optimal,villani2021topics}, $y_x = \Amat\Amat^\dagger x = T(x)$, where $T$ is the optimal transport map from $\rho_y^*$ to $\rho_y^\infty$. That is,
$$
\rho_y^\infty = (\Amat \Amat^\dagger)_\sharp \rho_y^*\,,
$$
which implies that $\rho_u^\infty = \Amat^\dagger_\sharp \rho_y^*$.
\end{proof}

Equation \eqref{eq:W2_over} shows that we have a simple form for the minimizer of the variational problem~\eqref{eqn:argmin_sde} and the equilibrium of the gradient flow equation~\eqref{eqn:argmin_sde3}  if $D$ is $W_2$ and $\mathsf{A}$ is tall skinny with full column rank. Just as $\Amat^\dagger y^*$ provides the optimizer in the deterministic case (see~\eqref{eq:det_over}), the minimizer in the stochastic context is its vanilla extension $\mathsf{A}^\dagger_\sharp\rho_y^*$ under the $W_2$ case, also the only equilibrium of the Wasserstein gradient flow under conditions.

\section{Numerical Examples}\label{sec:numerics}
This section presents a few numerical inversion examples using the proposed particle method~\eqref{eqn:particle_method}. Throughout the section, we use the KL divergence~\eqref{eqn:KL_obj} as the objective function and the $W_2$ metric to determine the geometry. {In all examples in this section, we set $\epsilon = 0.5$ as the hyperparameter in the density estimation step~\eqref{eqn:particle_KL}. The time step size to discretize the gradient dynamics~\eqref{eqn:particle_KL} is chosen using the Armijo backtracking line search.}
 
\subsection{Linear Push-Forward Map}
We first present examples with the linear push-forward map $y = \mathcal{G}(u)=\mathsf{A}u$, which we theoretically studied in~\Cref{sec:theory}. Given the matrix $\mathsf{A}$ is fully-, under-, or over-determined, we show that there are different phenomena in inversion using the gradient flow approach~\eqref{eqn:GF}. 

\subsubsection{Fully-Determined Case}
First, we consider a fully-determined case. Let $\mathsf{A} = \text{diag}([2,0.75])$, and the true parameter $u\sim \mathcal{N}(0,I)$. As a result, the reference data  $y\sim\mathcal{N}(0,\mathsf{A}\mathsf{A}^\top)$, represented by an empirical distribution with $1000$ i.i.d.~particles. The initial guess for the parameter is the uniform distribution $\mathcal{U}[-3,3]^2$, which is also represented by $1000$ i.i.d.~samples. We then follow the particle method~\eqref{eqn:particle_method} to implement the gradient flow equation~\eqref{eqn:GF}.  The convergence results after $30$ iterations are shown in~\Cref{fig:Full2to2}. We have recovered both the parameter distribution and the data distribution well.

\begin{figure}
    \centering
    \includegraphics[width = 0.8\textwidth]{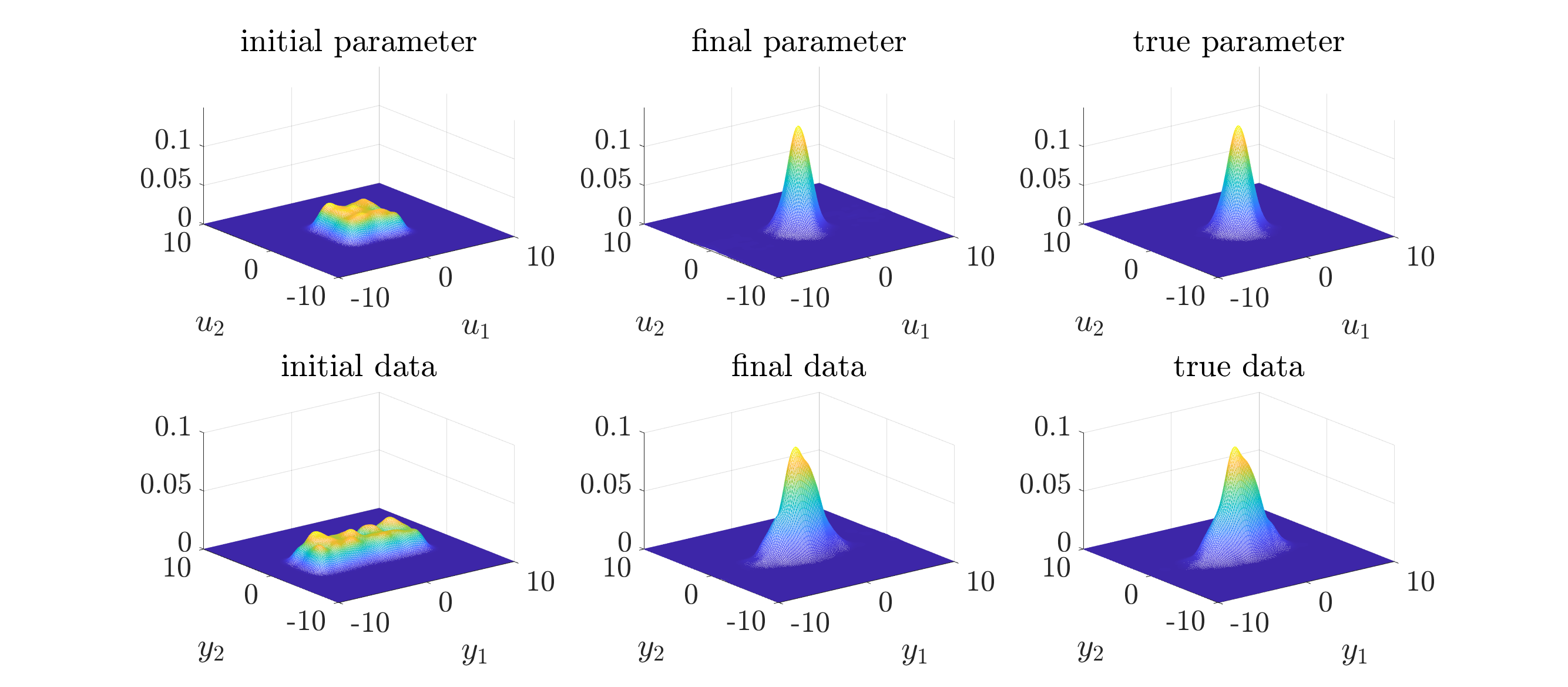}
    \caption{Parameter and data distributions in the fully-determined case with the map $T(x) = \mathsf{A}x$ where $\mathsf{A} = \text{diag}([2,0.75])$.}
    \label{fig:Full2to2}
\end{figure}

\begin{figure}
    \centering
   \subfloat[Parameter distribution with initial guess $u^1$]{ \includegraphics[width = 0.8\textwidth]{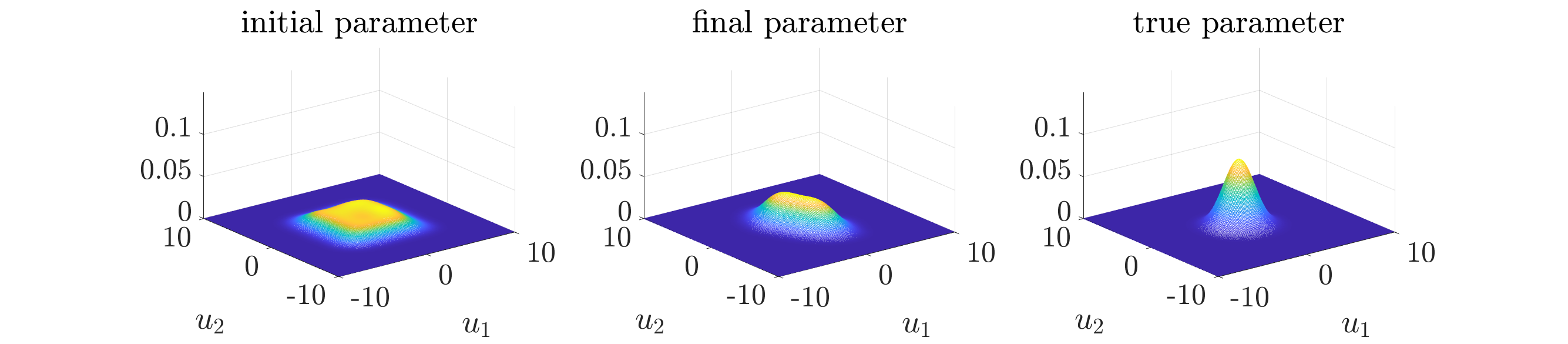}}\\
   \subfloat[Parameter distribution with initial guess $u^2$]{ \includegraphics[width = 0.8\textwidth]{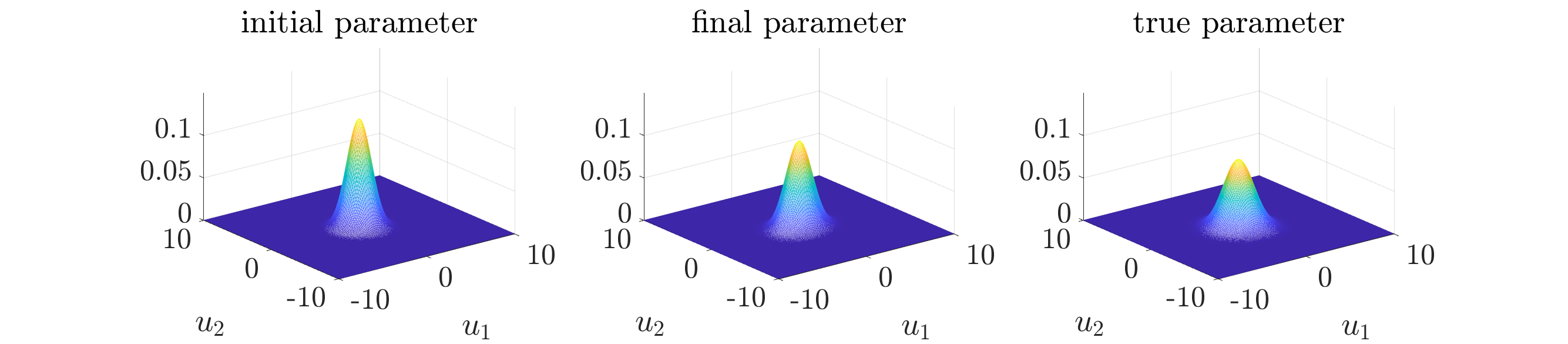}}\\
      \subfloat[Data with initial guess $u^2$]{ \includegraphics[width = 0.4\textwidth]{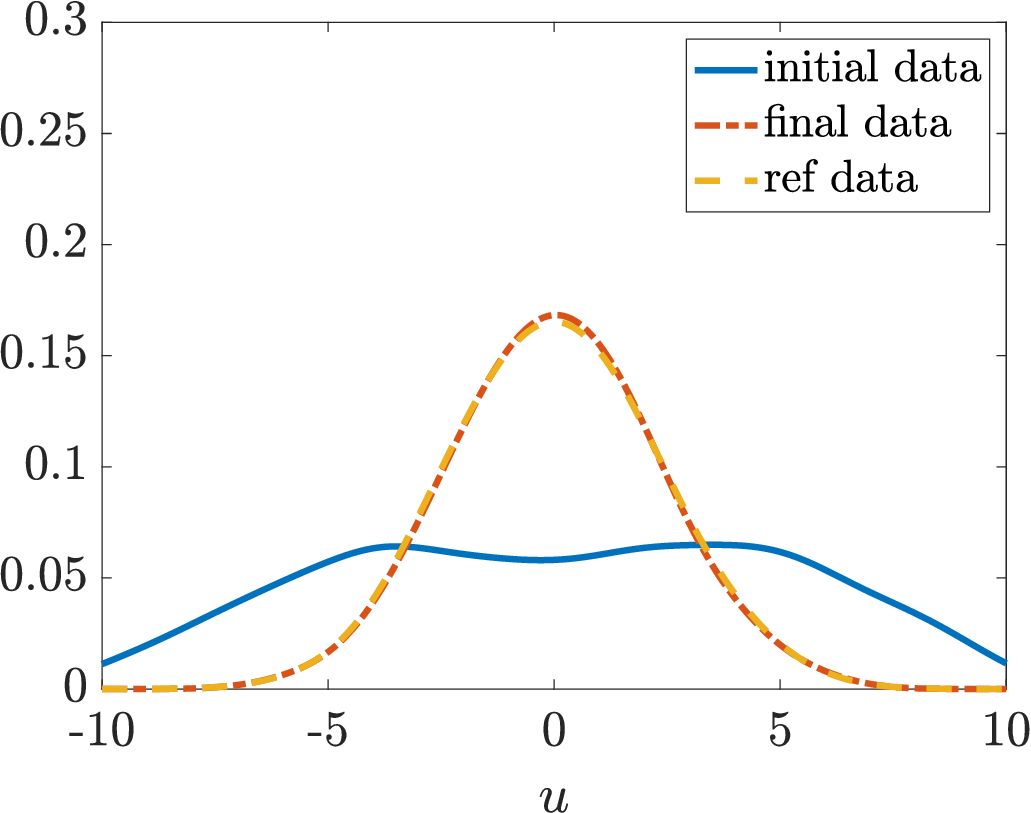}}
      \hspace{0.1cm}
   \subfloat[Data with initial guess $u_2$]{ \includegraphics[width = 0.4\textwidth]{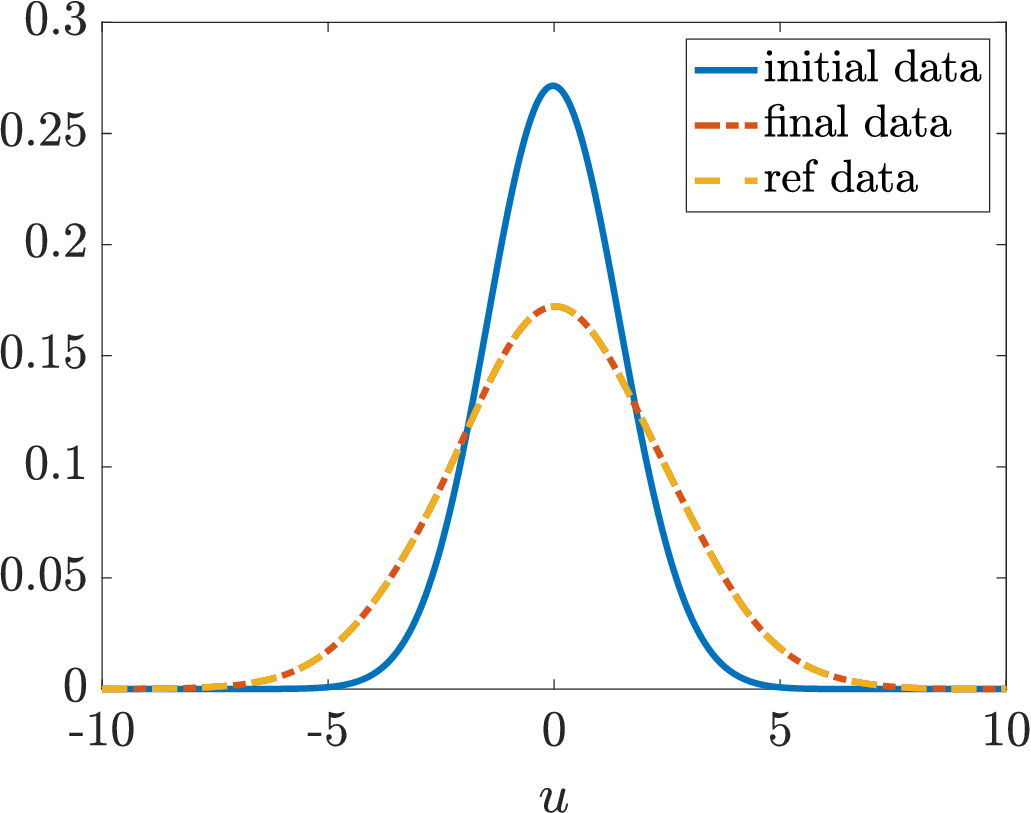}}
    \caption{Under-determined case under two  initial distributions, $u^1$ and $u^2$, where the map $\mathcal{G}(u) =  \mathsf{A}u$, $\mathsf{A} = [2, 0.75]$. The reference data distribution is computed using the true parameter distribution.}
    \label{fig:under2to1}
\end{figure}

\subsubsection{Under-Determined Case}
Next, we consider that  $\mathsf{A} = [2, 0.75]$, which is under-determined. The true parameter distribution is $\mathcal{N}(0,I)$ and the reference  data distribution $y\sim\mathcal{N}(0,\mathsf{A}\mathsf{A}^\top=4.5625)$, which we use $3000$ i.i.d.~samples to represent. Since $\mathsf{A}$ is under-determined, based on our analysis in~\Cref{sec:theory}, there are infinitely many solutions $u$, the same as the deterministic case. Since we use the gradient flow formulation~\eqref{eqn:GF}, the algorithm can only find one of the solutions, which is determined by the initial condition. We then demonstrate this through numerical examples.

We consider two different initial distributions $u^1$ and $u^2$, as shown in~\Cref{fig:under2to1}. We also use $3000$ particles to implement the Wasserstein gradient flow of minimizing the KL divergence between the reference and the data distribution computed from the current iterate of the parameter distribution. Although the gradient flow started from neither of the initial distributions converges to the \textit{true} parameter distribution from which we generate the data, their corresponding data distributions match the reference and successfully achieve data fitting. This again demonstrates the intrinsic non-uniqueness of the inverse problem when the linear push-forward map $\mathsf{A}$ is under-determined.

\subsubsection{Over-Determined Case}
Next, we show the over-determined case with $\mathsf{A} =[2, 1]^\top$. We set the true parameter distribution to be $\mathcal{N}(0,1)$ and choose a reference data distribution polluted by random noise and thus is not in the range of the forward push-forward map. Inversion in this scenario is similar to the least-squares method in the deterministic case; see~\Cref{sec:theory}. We use $3000$ samples to represent the reference data distribution and $3000$ particles to implement the gradient flow method~\eqref{eqn:GF}. 

In~\Cref{fig:over1to2}, we plot the initial and final converged parameter distributions and the corresponding data distributions pushed forward by the forward map from those parameter distributions. We also show the true parameter distribution and the reference data distribution for comparison. The final data distribution from the recovered parameter distribution does not fit the reference data entirely. However, their marginal distributions along $y_\Amat$ (orthogonal projection of $y$ over the column space of $\Amat$) match exactly. This verifies our result in~\Cref{thm:over_determin_stoc}. 

\begin{figure}
    \centering
    \includegraphics[height = 3cm]{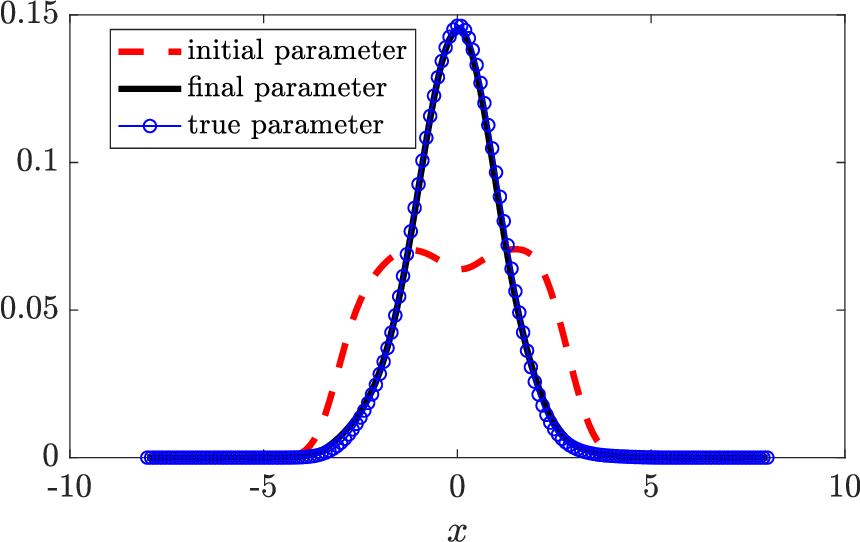} \includegraphics[height = 3.5cm]{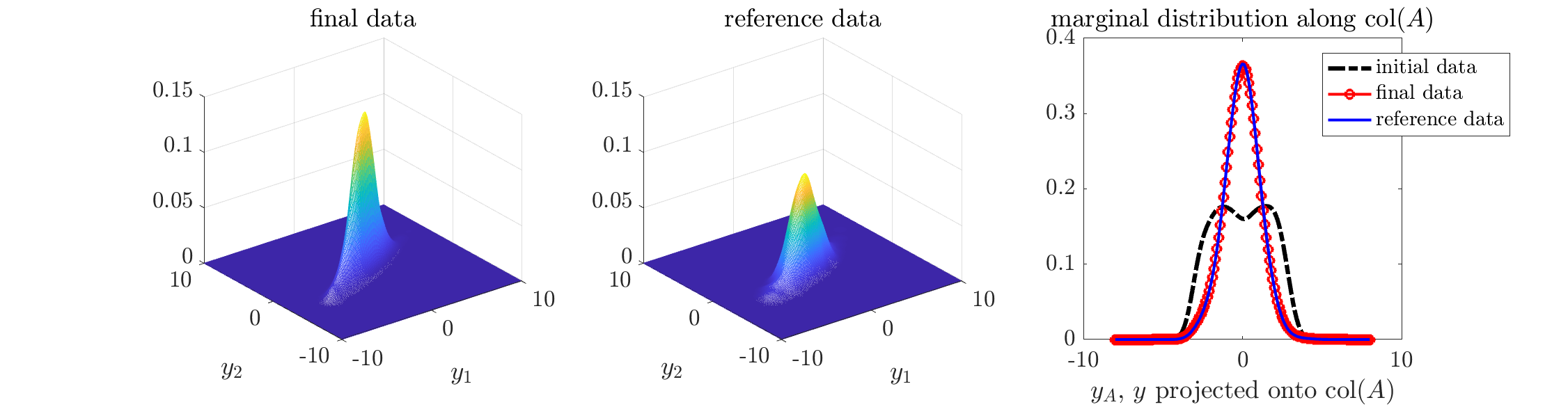}
    \caption{Over-determined case with the map $T(x) = \mathsf{A}x$ where $ \mathsf{A} = [2, 1]^\top$. Although the final recovered data distribution does not fit the reference entirely, their marginal distributions on $y_A$ match very well, as proved in~\Cref{thm:over_determin_stoc}.}
    \label{fig:over1to2}
\end{figure}

\subsection{An Inverse Problem Example}
Many inverse problems can be solved in our framework. Let us first consider a one-dimensional elliptic boundary value problem, a test case considered in~\cite{herty2018kinetic,garbuno2020interacting}:
\begin{alignat}{2} \label{eq:1D-eit}
 -\left(\exp(u_1)  p'(x) \right)' &= 1,\quad && x\in[0,1],\\
 p(0) &= 0, && p(1) = u_2. \nonumber
\end{alignat}
Its analytic solution can be written as 
\begin{equation}\label{eq:true sol 1D}
p(x) = u_2 x + \exp(-u_1)\left( - \frac{x^2}{2}  + \frac{x}{2}\right),
\end{equation}
which shows the log stability of this particular setup (parameter change in response to the PDE solution change). In the setup of~\cite[Eqn.~(4.4)]{garbuno2020interacting}, the authors consider the forward model mapping the two independent scalar coefficients $u_1$ and $u_2$ to the observed data $y_1:=p(x_1)$ and $y_2:=p(x_2)$ where $x_1 = 0.25$ and $x= 0.75$. 

Following the same setup, consider $u_1$ and $u_2$ are scalar-valued random variables, our observations $[y_1, y_2]^\top$ are also random in nature. We represent $\rho_y^\ast$ using $N=5000$ samples in this test and run the gradient flow simulation using $M=5000$ simulated particles. Two settings are considered. 
\begin{enumerate}[label=(\arabic*)]
    \item The true parameters $u_1\sim \mathcal{N}(0,0.5)$, and $u_2\sim \mathcal{U}([0,2])$. The initial guess $u_1^{(0)}, u_2^{(0)}$ both follow $\mathcal{N}(0,2)$.
    % The number of particles $M = N = 5000$.
    \item $u_1\sim  \mathcal{N}(-1.5,0.5)$, and $u_2\sim \mathcal{U}([0,2])$. The initial guess $u_1^{(0)}\sim \mathcal{U}([-3,-1])$ and $u_2^{(0)}\sim \mathcal{U}([0,2])$.  
\end{enumerate}
Since the analytic solution~\eqref{eq:true sol 1D} suggests stronger sensitivity of data on the negative values of $u_1$, we expect a better stability of the inverse problem in setting (2), given its ground-truth taking on a negative value with high probability. This is indeed what we observe from the experiments; see~\Cref{fig:EIT-1D-1,fig:EIT-1D-2}. In both cases, the data distributions are matched very well. The recovered parameter distribution, however, demonstrates different features: It is visually far away from the truth in setting (1) but is in much better agreement with the ground truth in setting (2).

\begin{figure}
    \centering    \includegraphics[width = 0.8\textwidth]{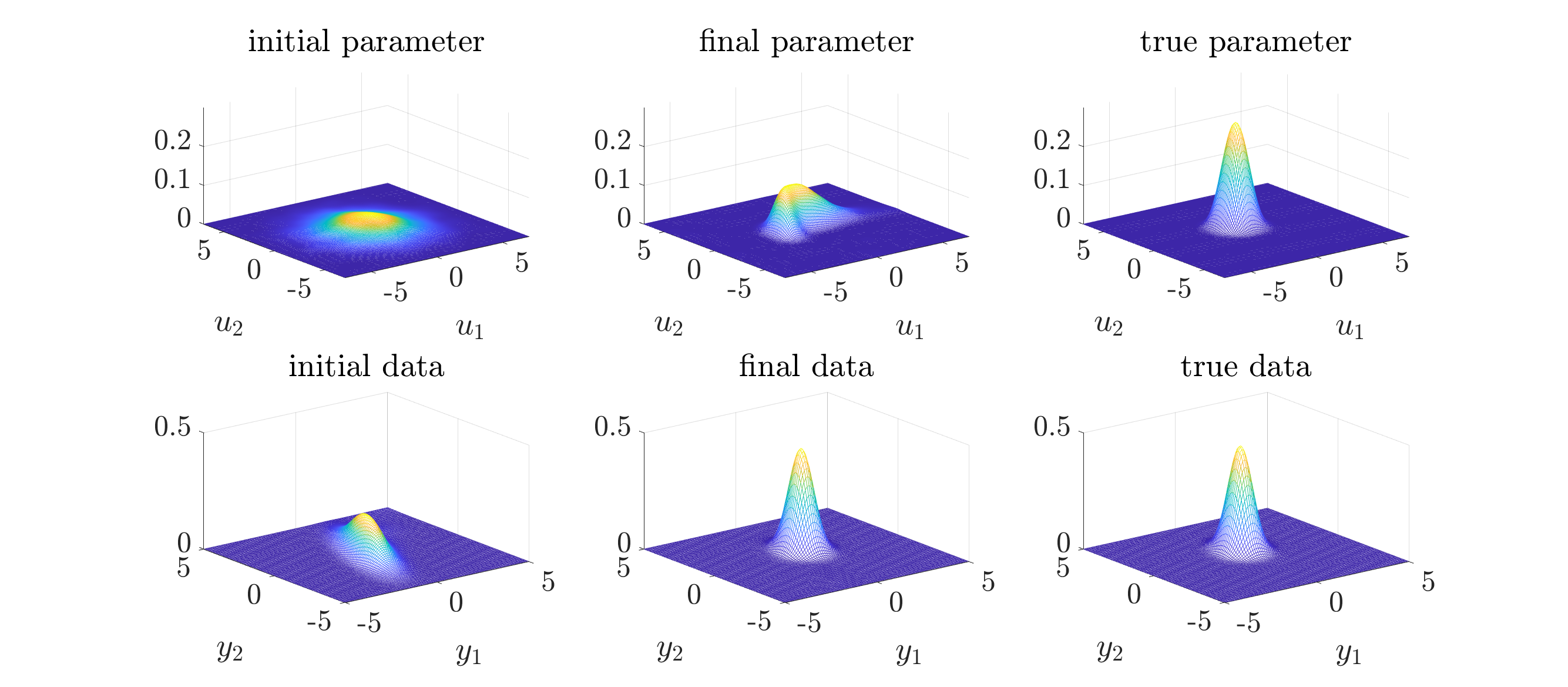}
    \caption{Numerical inversion based on the 1D diffusion equation~\eqref{eq:1D-eit} with setting (1).}\label{fig:EIT-1D-1}
\end{figure}

\begin{figure}
    \centering
   \includegraphics[width = 0.8\textwidth]{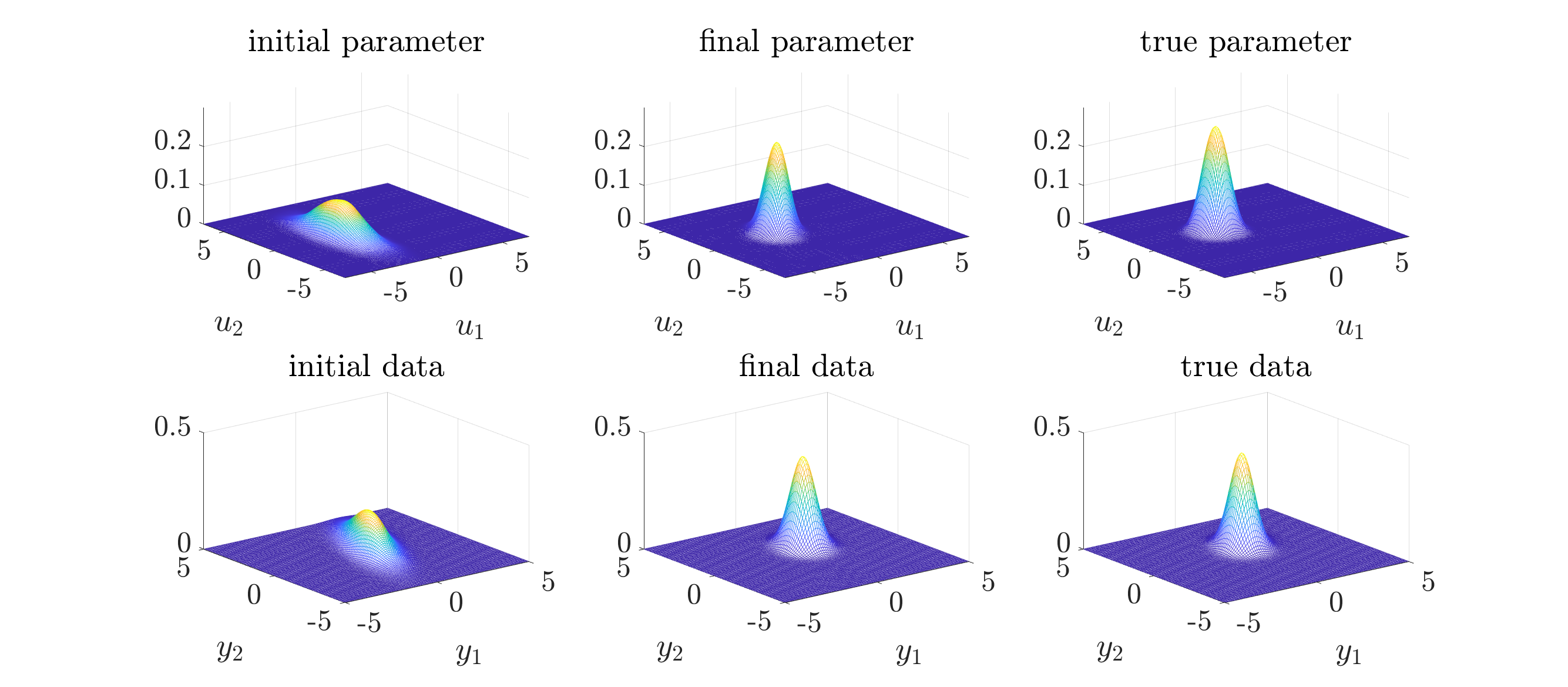}
    \caption{Numerical inversion based on the 1D diffusion equation~\eqref{eq:1D-eit} with setting (2).}\label{fig:EIT-1D-2}
\end{figure}

The higher-dimensional version of~\eqref{eq:1D-eit} becomes
\begin{alignat}{2}
-\nabla \cdot \left(a({\bf x})  \nabla  p({\bf x}) \right) &= f({\bf x}), \qquad \qquad     && {\bf x} \in D, \label{eq:EIT-2D}\\
p({\bf x}) &= \sin(2 x_1 \pi) \cos (4 x_2\pi), \qquad && {\bf x} = (x_1,x_2) \in \partial D. \nonumber
\end{alignat}
We consider $D = [0,1]^2$ and $a({\bf x}) = \exp\left( u_1 \phi_1({\bf x}) + u_2 \phi_2({\bf x}) \right)$ where $\phi_1({\bf x}) = \frac{10}{9+\pi^2} \cos(\pi x_1)$ and $\phi_2({\bf x}) = \frac{10}{9+2\pi^2} \cos(\pi (x_1+x_2))$, following a similar setup in~\cite[Sec.~4.4]{garbuno2020interacting}. Two observed empirical distributions $y_1 = p({\bf x}_1)$ and $y_2 = p({\bf x}_2)$ are used to perform the inversion where the receivers are located at ${\bf x}_1 = (0.25,1)^\top$ and ${\bf x}_2 = (1,0.5)^\top$. The true distributions are $u_1 \sim \mathcal{N}(1,1)$ and $u_2\sim \mathcal{U}([0,1])$ while the initial distribution $u_1^{(0)},u_2^{(0)}\sim \mathcal{U}([-2.5,2.5])$. We use $N = 5000$ particles for the observed empirical distribution and $M=1000$ for the simulation. Both the parameter reconstruction and its generated data distribution are close to the truth. The smooth-basis parameterization of $a({\bf x})$ significantly improves the well-posedness of the problem. The convergence history of the objective function (i.e., the KL divergence) for these three cases is plotted in Figure~\ref{fig:EIT-obj}.

\begin{figure}
    \centering
   \includegraphics[width = 0.8\textwidth]{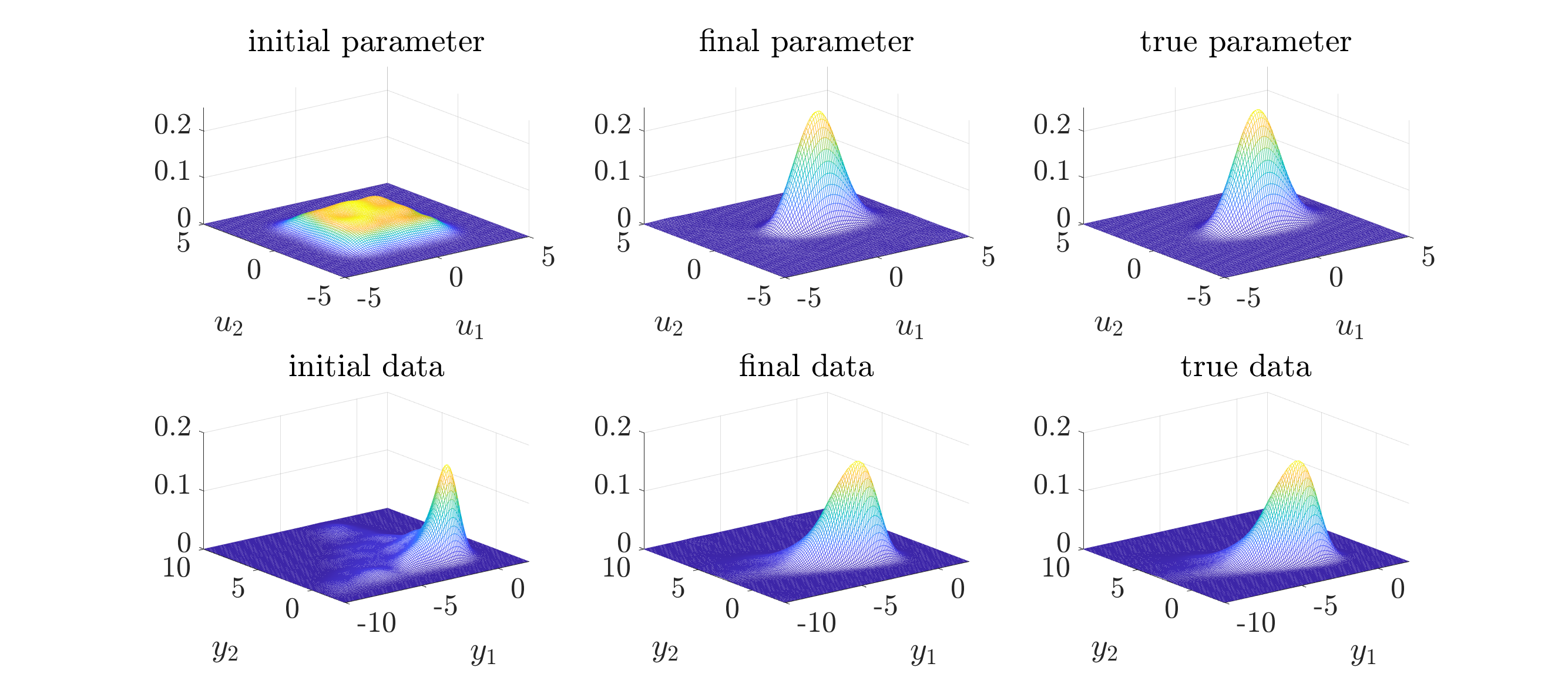}
    \caption{Numerical inversion based on the 2D diffusion equation~\eqref{eq:EIT-2D}.}\label{fig:EIT-2D-1}
\end{figure}

\begin{figure}
\centering
\subfloat[For~\Cref{fig:EIT-1D-1}]{\includegraphics[width = 0.3\textwidth]{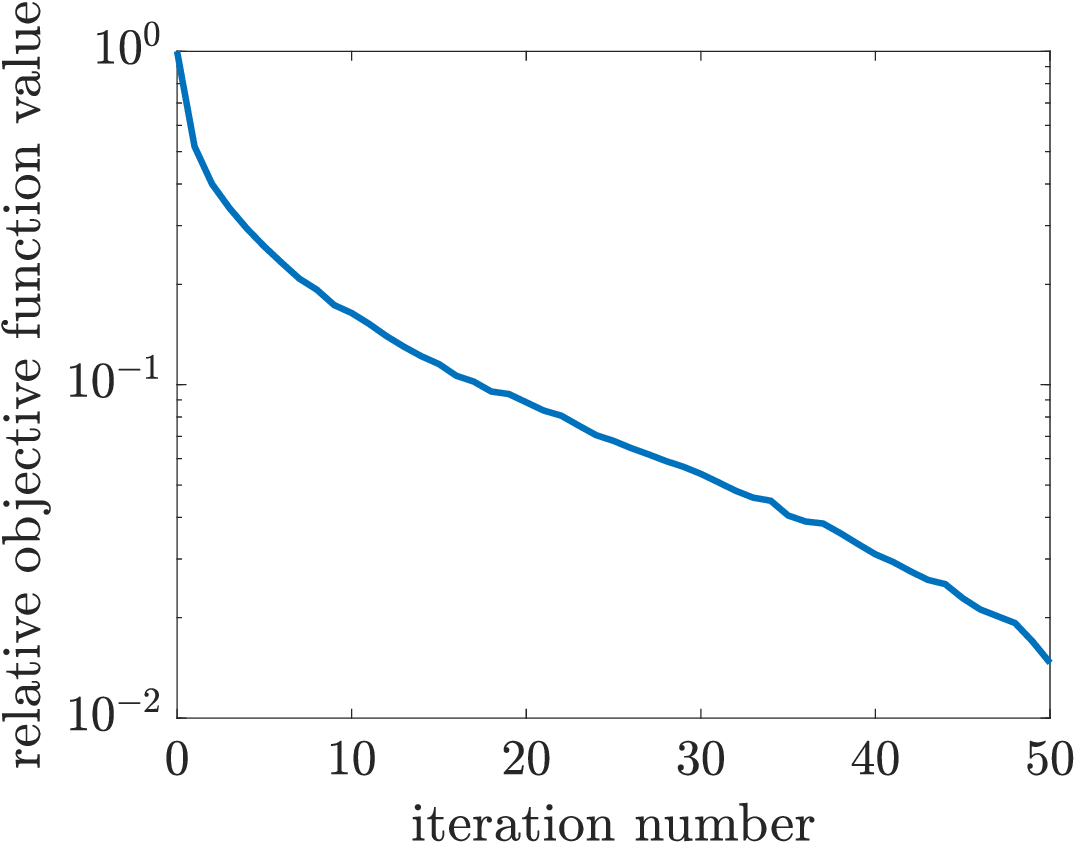}}
\subfloat[For~\Cref{fig:EIT-1D-2}]{\includegraphics[width = 0.3\textwidth]{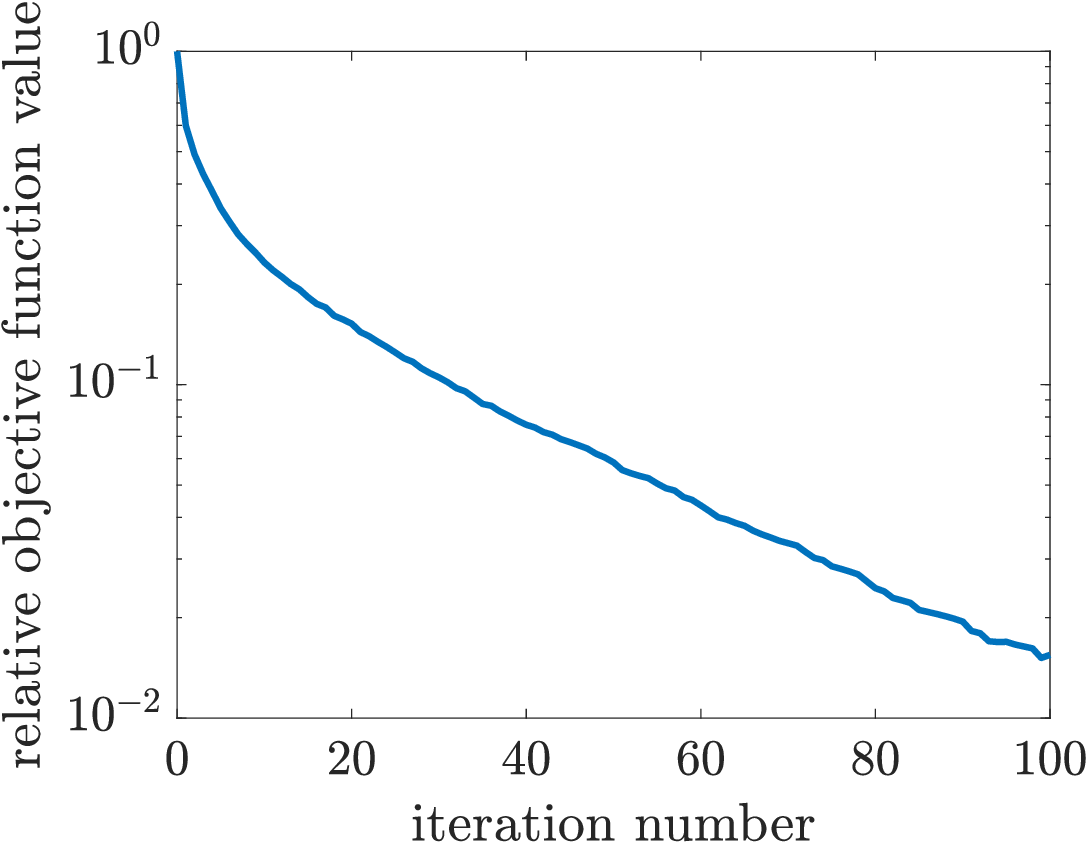}}
\subfloat[For~\Cref{fig:EIT-2D-1}]{\includegraphics[width = 0.3\textwidth]{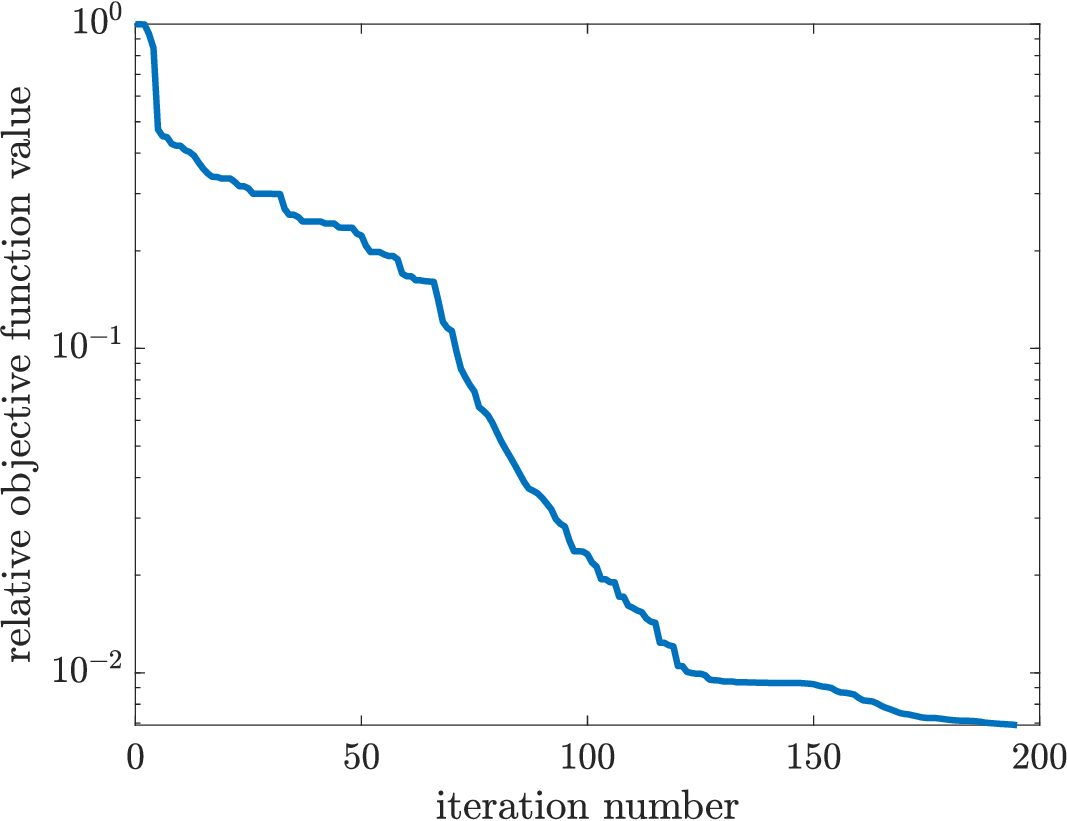}}
\caption{The objective function decay for the tests shown in~\Cref{fig:EIT-1D-1,fig:EIT-1D-2,fig:EIT-2D-1}.}
\label{fig:EIT-obj}
\end{figure}

\section{Conclusion and Discussions}\label{sec:conclude}
While most research in inverse problems focuses on deterministic unknowns, many real-world problems are inherently stochastic, introducing random variations in the parameters to be inferred. This necessitates a paradigm shift in research from optimizing a single, deterministic value to characterizing the full probability distribution that governs these parameters.

This paper presents a framework for conducting stochastic inverse problems. For linear push-forward maps, we also discuss the well-posedness theory both for the formulation and for the gradient-flow algorithm. Particle-based solvers are designed to simulate the gradient flow for finding the optimal probability measure.

The proposed approach shares many traits similar to Bayesian inversion. While it is true that both formulations return probability distributions of the unknown parameter, there is a stark difference between the two: the sources of randomness are different. In Bayesian inference, it is assumed there is randomness in the prior knowledge (encoded in the prior distribution) and there is measurement error (encoded in the likelihood function). On the contrary, our formulation assumes the data is devoid of any noise, attributing randomness solely to the parameters. To incorporate the randomness in the measurement error, we need to reformulate the problem as a probability over probability measure space, and seek for $\textrm{Law}\{\rho_u\}$, instead of $\rho_u$ itself. This would be a much more intricate problem, and we leave it to future research.

The work is just the beginning of our research endeavor to explore stochastic inverse problems. The rich geometries of probability spaces present a unique opportunity for computational solver designing. We expect different combinations of the data discrepancy and the metric deployed to measure probability distances would present different features of the problem.

\section*{Acknowledgement}
The authors thank Prof.~Youssef Marzouk, Prof.~Levon Nurbekyan,  Prof.~Kui Ren and Prof.~Andrew Stuart for all the valuable discussions.  We also thank the anonymous referees for their time and helpful suggestions.

\appendix
\section{Proof of~\Cref{prop:under_determine}} \label{app1}
\begin{proof}[Proof of~\Cref{prop:under_determine}]
The proof is simple algebra. We note that the source term in equation~\eqref{eqn:theta_ODE_determ} is always in the column space of $\mathsf{A}^\top$, hence the column space of $\mathsf{U}$, which leads to the fact that
\begin{equation}\label{eqn:theta_deter_trajectory}
u(t)\in u_0+\text{span}\{\mathsf{U}\}\,.
\end{equation}
Since the convergence of gradient descent is achieved when $u(t)\in\mathcal{S}$, we have 
\[
u_\text{f}\in \big\{ u_0+\text{span}\{\mathsf{U}\} \big\} \cap \big\{ u^\ast+\text{span}\{\mathsf{U}^\perp\}\big\}\,.
\]
Let $\mathsf{U} = \begin{bmatrix} u_1 & \ldots  & u_n  \end{bmatrix}$ and  $\mathsf{U}^\perp=\begin{bmatrix} u_{n+1} & \ldots & u_m \end{bmatrix}$, where the column vector $u_i \in \mathbb{R}^m$ and $\|u_i\|_2 = 1$, $1\leq i \leq n$.  Then we have that
\[
u_\text{f}=u_0 + \sum_{i=1}^n\lambda_iu_i=u^\ast - \sum_{i=n+1}^m\lambda_iu_i\,,
\]
making $\lambda_i = u_i^\top (u^\ast-u_0)$, which finalizes to
\[
u_\text{f}=u_0+\sum_{i=1}^n u_i^\top(u^\ast-u_0)u_i=u^\ast-\sum_{i=n+1}^m u_i^\top(u^\ast-u_0)u_i\,.
\]
Equivalently,
$
u_\text{f}=u_0-\mathsf{U}\mathsf{U}^\top(u_0-u^\ast)=u^\ast-\mathsf{U}^\perp(\mathsf{U}^\perp)^\top(u^\ast-u_0)$. Then the result follows from the following identity 
\[
\mathsf{U} \mathsf{U}^\top + \mathsf{U}^\perp (\mathsf{U}^\perp)^\top = \mathsf{I}\,,
\]
which follows from $\mathsf{U}^\top \mathsf{U}^\perp = 0$ and $(\mathsf{U}^\perp)^\top \mathsf{U} = 0$. Furthermore, given $y$ and $\tilde{y}$ in~\eqref{eqn:def_A_ex}, and the fact that the row space for $\mathsf{A}$ and $\tilde{\mathsf{A}}$ are $\mathsf{U}$ and $\mathsf{U}^\perp$, respectively, we have $y_\text{f} = \mathsf{A}u_\text{f} = y^\ast$ and $\tilde{y}_\text{f}=\tilde{\mathsf{A}} u_0$.
\end{proof}
The proof is rather straightforward. We quickly comment on its geometric interpretation. The result essentially is looking for the intersection of two sets. One is defined by the trajectory~\eqref{eqn:theta_deter_trajectory}, and the other is defined by the equilibrium~\eqref{eqn:equi_set_determ}. The intersection point is unique, with its $\mathsf{U}$ component determined by $u^\ast$ and the $\mathsf{U}^\perp$ component determined by the initial guess $u_0$. Note also that although $u^*$ is not unique, its projection to the column space of $\mathsf U$ is unique. Indeed, since $\mathsf{A} u^* =y^*$, we have $\mathsf{U}\mathsf{U}^\top u^\ast = \mathsf{U}\mathsf{S}^{-1}\mathsf{V}^\top y^*$, which is purely determined by $\mathsf{A}$ and the given $y^*$.

\bibliographystyle{siam}
\bibliography{ref}

\end{document}